\documentclass[reqno]{amsart}
\usepackage[normalem]{ulem}
\usepackage{hyperref} 
\usepackage{color}
\usepackage{latexsym}
\usepackage{graphicx}
\usepackage{mathrsfs}
\usepackage{amssymb}
\usepackage{amsfonts}
\usepackage{amssymb,amsmath,amsthm,tikz}
\usepackage{tikz-cd}

\hypersetup{colorlinks = true, urlcolor = black}
\DeclareMathOperator{\Erf}{Erf}

\renewcommand{\Re}{\text{Re}}
\renewcommand{\Im}{\text{Im}}

\newcommand{\szego}{Szeg\"o }
\newcommand{\Kahler}{K\"ahler }
\newcommand{\kahler}{K\"ahler }
\newcommand{\Sjostrand}{Sj\"ostrand }
\newcommand{\Sj}{Sj\"ostrand }

\newcommand{\gfrak}{\mathfrak{g}}

\newcommand{\C}{\mathbb{C}}
\newcommand{\E}{\mathbb{E}}

\newcommand{\CP}{\mathbb{CP}}

\newcommand{\R}{\mathbb{R}}

\newcommand{\Z}{\mathbb{Z}}

\renewcommand{\P}{\mathbb{P}}

\newcommand{\T}{\mathbf{T}}

\newcommand{\acal}{\mathcal{A}}

\newcommand{\fcal}{\mathcal{F}}

\newcommand{\hcal}{\mathcal{H}}
\newcommand{\ical}{\mathcal{I}}

\newcommand{\lcal}{\mathcal{L}}

\newcommand{\ocal}{\mathcal{O}}
\newcommand{\pcal}{\mathcal{P}}

\newcommand{\scal}{\mathcal{S}}

\newcommand{\wt}{\widetilde}
\newcommand{\wb}{\overline}

\newcommand{\bma}{\begin{bmatrix}}
\newcommand{\ema}{\end{bmatrix}}
\newcommand{\baa}{\begin{align*}}
\newcommand{\eaa}{\end{align*}}
\newcommand{\bea}{\begin{eqnarray*} }
\newcommand{\eea}{\end{eqnarray*} }
\newcommand{\bee}{\begin{eqnarray} }
\newcommand{\eee}{\end{eqnarray} }
\newcommand{\be}{\begin{equation} }
\newcommand{\ee}{\end{equation} }
\newcommand{\bp}{\begin{prop}}
\newcommand{\ep}{\end{prop}}
\newcommand{\bt}{\begin{theorem}}
\newcommand{\et}{\end{theorem}}
\newcommand{\bpf}{\begin{proof}}
\newcommand{\epf}{\end{proof}}
\newcommand{\bl}{\begin{lem}}
\newcommand{\el}{\end{lem}}
\newcommand{\bc}{\begin{cor}}
\newcommand{\ec}{\end{cor}}
\newcommand{\bd}{\begin{defn}}
\newcommand{\ed}{\end{defn}}
\newcommand{\bcs}{\begin{cases}}
\newcommand{\ecs}{\end{cases}}
\newcommand{\bex}{\begin{example}}
\newcommand{\eex}{\end{example}}
\newcommand{\brem}{\begin{rem}}
\newcommand{\erem}{\end{rem}}

\renewcommand{\ss}{\subsection}

\newcommand{\pa}{\partial}
\newcommand{\ot}{\otimes}
\newcommand{\half}{\frac{1}{2}}
\renewcommand{\d}{\partial}
\newcommand{\dbar}{\bar\partial}
\newcommand{\ddbar}{\partial\dbar}
\newcommand{\RM}{\backslash}

\newcommand{\inv}{^{-1}}

\newcommand{\la}{\lambda}

\newcommand{\hPi}{\h \Pi}

\newcommand{\h}{\hat} 

\newcommand{\lan}{\langle}
\newcommand{\ran}{\rangle}

\def\Xint#1{\mathchoice
{\XXint\displaystyle\textstyle{#1}}%
{\XXint\textstyle\scriptstyle{#1}}%
{\XXint\scriptstyle\scriptscriptstyle{#1}}%
{\XXint\scriptscriptstyle\scriptscriptstyle{#1}}%
\!\int}
\def\XXint#1#2#3{{\setbox0=\hbox{$#1{#2#3}{\int}$ }
\vcenter{\hbox{$#2#3$ }}\kern-.6\wd0}}

\def\dashint{\Xint-}

\newtheorem{Theo}{{\sc Theorem}}
\newtheorem{theo}{{\sc Theorem}}[section]
\newtheorem{cor}[theo]{{\sc Corollary}}

\newtheorem{lem}[theo]{{\sc Lemma}}
\newtheorem{prop}[theo]{{\sc Proposition}}
\newtheorem{defn}[theo]{{\sc Definition}}
\newtheorem{rem}{{\sc Remark}}
\newenvironment{example}{\medskip\noindent{\it Example:\/} }{\medskip}

\title[Interface asymptotics of partial  Bergman kernels]{Interface asymptotics of partial  Bergman kernels on $S^1$-symmetric \Kahler manifolds}

\author{Steve Zelditch and Peng Zhou}
\address{Department of Mathematics, Northwestern  University, Evanston, IL 60208, USA}

\email{zelditch@math.northwestern.edu}

\thanks{Research partially supported by NSF grant  DMS-1541126
and by the Stefan Bergman trust  .}

\begin{document}

\begin{abstract} This article is concerned with asymptotics of equivariant
Bergman kernels and partial Bergman kernels for polarized projective \kahler
manifolds invariant under a Hamiltonian holomorphic $S^1$ action. Asymptotics   of partial Bergman kernel are obtained  in the allowed region $\acal$ resp.  forbidden region $\fcal$,  generalizing 
results of Shiffman-Zelditch,  Shiffman-Tate-Zelditch  and Pokorny-Singer   for toric \kahler manifolds. The main result  gives   scaling asymptotics of equivariant Bergman kernels and partial Bergman kernels in the transition region around the interface $\partial \acal$,  generalizing recent work of Ross-Singer on partial Bergman kernels,  and refining the Ross-Singer
transition asymptotics to apply to equivariant Bergman kernels.

\end{abstract}

\maketitle  

   This article is concerned with the asymptotics of  {\it partial Bergman kernels} for 
  positive  Hermitian holomorphic line bundles $(L,h) \to (M, \omega)$ over a \kahler manifold of complex dimension $m$  carrying a Hamiltonian holomorphic $S^1$   action 
    \[ \exp t \; \xi_H: \T \times M \to M, \;\;\; \iota_{\xi_H} \omega = d H,
  \;\; \exp t \xi_H (z): = e^{2\pi i t}  z, \]
where $H: M \to P_0 := H(M) \subset \R$ is the Hamiltonian and $\xi_H$ is its Hamilton vector field.  
  The $\T$-action\footnote{
 We denote it by $\T$ rather than by $S^1$ because we use that  notation for a different circle action on $L^*$. We also use the terms Bergman kernel and \szego kernel interchangeably.} preserves the data
 $(L,h)$ and can be `quantized' to give a unitary representation of $\T$  \begin{equation} \label{Ukt} U_k(\theta) = e^{i k \theta \hat{H}_k}: \T \times H^0(M, L^k) \to H^0(M, L^k) \end{equation}   on the spaces $H^0(X, L^k)$ of holomorphic sections of the  tensor powers $L^k $, equipped with the $L^2$ norm ${\rm Hilb}_{h^k}$   induced
by the Hermitian metric $h$.
The  self-adjoint  generator
 of $U_k(\theta)$  is denoted by 
\begin{equation} \label{hatHkdef} \hat{H}_k := H + \frac{i}{2\pi k}  \nabla_{\xi_H} :  H^0(M, L^k) \to H^0(M, L^k), \end{equation}  where $\nabla_{\xi_H} s$ is
the covariant derivative of a section $s$ and $H s$ is the product of $s$
with $H$ \cite{Ko,GS}.  When $\xi_H $ generates  
a holomorphic $\T$ action, $\hat{H}_k$ preserves holomorphic sections and coincides with the Toeplitz operator
\[ \hat{H}_k s = \Pi_{k} \hat{H}_k \Pi_{k} s, \;\;\; s\in H^0(M, L^k) \]
with principal symbol $H$  (see  \S \ref{LIFT}).
Here,
\[ \Pi_{k}: L^2(M, L^k) \to H^0(M, L^k) \]
is the orthogonal projection (or \szego-Bergman kernel).

We define the eigenspaces of $\hat{H}_k$ (= the weight spaces of the $S^1$ action) by \begin{equation} \label{EIG} V_{k} (j)= \{s \in H^0(M, L^k): U_k(\theta)  s = e^{i j \theta} s \} =  \{s \in H^0(M, L^k): \hat{H}_k  s = \frac{j}{k} \; s \}; \end{equation}  it is known that $V_k(j) \not= \{0\}$ if and only if $\frac{j}{k} \in P_0 = H(M)$, and  their dimensions 
  have been  computed in articles on 
  ``quantization commutes with reduction'' \cite{GS}. In  Lemma \ref{HMLEM} we show that $H(M) = [0, a]$ for a positive integer $a$ which is equal to the symplectic area of a generic $\C^*$ orbit. 
 We define the    associated weight space projections (termed {\it equivariant Bergman kernels})  
  \begin{equation} \label{SPECPROJ} \Pi_{k, j}(z,w) : L^2(M, L^k) \to V_{k}(j). \end{equation} 
  These equivariant Bergman kernels are the smallest components of  
  the full Bergman kernel (or \szego projector)
  \begin{equation} \Pi_{k}(z,w) = \sum_{j: \frac{j}{k} \in P_0} \Pi_{k,j} : L^2(M, L^k) \to 
  H^0(M, L^k) \end{equation}
  to possess strong asymptotic expansions when $\frac{j}{k} \to E$ for
  some value $E$ of $H$.  The norm  contraction of $\Pi_{k, j}(z,z)$ on the
  diagonal is denoted by $\Pi_{k,j}(z)$ and is called the equivariant density
  of states.\footnote{The norm contraction of any kernel $K(z,w)$ on the diagonal is denoted $K(z)$.}  As proved in Theorems \ref{E} and \ref{EQUIVINT}, the normalized equivariant density of states
  $k^{-m + \half} \Pi_{k, j}(z)$ resembles a Gaussian bump concentrated on the energy
  level $H^{-1}(E) $ in the sense of being essentially equal to $1$ on $H^{-1}(\frac{j}{k})$ and having ``Gaussian decay'' $e^{- k b_E(z)}$ away from $H^{-1}(\frac{j}{k})$ along gradient
  lines $\sigma \to e^{-\sigma/2}\cdot z$ of $H$,  where $b_E$ is  defined by
  \eqref{bE}, and is like distance-squared to the hypersurface $H^{-1}(E)$. This is the analogue for $S^1$ actions of the result of \cite{STZ} showing that joint eigensections $z^{\alpha}$ of the torus action of a toric \kahler manifold are  Gaussian-like
  bumps centered on the tori $\mu^{-1}(\alpha)$ (the inverse image of $\alpha \in \Z^m$ under the moment map $\mu$), a fact also used in \cite{PS,RS}.

  The partial Bergman kernels of the title 
   are projectors   \begin{equation}\label{PkPdef}
   \Pi_{k, P} (z,w): = \sum_{j: \frac{j}{k} \in P} \Pi_{k, j}(z,w). \end{equation} onto subspaces  
  \begin{equation}\label{SkP} \scal_{k, P}: = \bigoplus_{j: \frac{j}{k} \in P}  V_k(j) \subset H^0(M, L^k) \end{equation}
  corresponding to   proper sub-intervals
  $P \subset P_0 = H(M)$. 
  They behave like  sums of Gaussian bumps centered at the the inverse images $H^{-1}(\frac{j}{k})$ of the ``lattice points'' $\frac{j}{k} \in P$. 

   The main problem is to relate the asymptotic properties of $\Pi_{k, P}(z,w)$ to the
   geometry of the Hamiltonian flow of $H$ and its complexification as a $\C^*$
   action. The analogous problem for toric \kahler manifolds was 
   studied in \cite{ShZ}, with $P$ a sub-polytope of the Delzant moment polytope of $(M, \omega)$.
 As in the toric case, we prove in Theorem \ref{thm:bulk} the norm contraction $\Pi_{k, P}(z)$ of
   $\Pi_{k, P}(z,z)$  has standard asymptotics in the {\it allowed region}  $\acal_P$ and exponentially
  decaying asymptotics in the {\it forbidden region} $\fcal_P$ where
   \[ \acal_P : = \text{int} \{z \in M: H(z) \in P\}, \;\;\; \fcal_P := \text{int} ( M \backslash \acal_P). \]
  On the boundary, or ``interface'' $\partial \acal_P$,  Ross-Singer in \cite{RS} showed that  $k^{-m} \Pi_{k, P}(z)$ decreases from $\sim 1$ to 
   $\sim 0$ in a tube of radius $k^{-\half} $. In  the special case where the minimum set of
   $H$ is a complex hypersurface, Theorem 1.2 of \cite{RS} asserts that if $\sqrt{k}(H(z)- E)$ is bounded, 
then
   \begin{equation} \label{RSG} k^{-n}  \Pi_{k, (-\infty, E]}(z)= \frac{1}{\sqrt{2 \pi |\xi_H(z)|^2}} \int_{-\infty}^{\sqrt{k}(H(z)- E)} e^{-\frac{t^2}{2 |\xi_H(z)|^2}} d t + O(k^{-\half}). \end{equation}
   Here, $|\xi_H|$ is the norm of $\xi_H$ with respect to the \kahler metric $\omega$.  The integral on the right is an incomplete Gaussian integral
   closely related  to the error function
   ${\rm erf}(x) = \frac{1}{\sqrt{\pi}} \int_{-x}^x e^{-t^2} dt$, which   is odd and smoothly interpolates between $-1$ at $-\infty$ to $1$ at $+ \infty$. The right side above involves  the slight modification   ${\rm Erf}(x) = \frac{1}{\sqrt{2 \pi}} \int_{-\infty}^x e^{-t^2/2} dt, $  which interpolates between $0$ at $-\infty$ and $1$ at $+ \infty$. It often arises in interface problems involved in packing quantum states in a domain (see e.g. \cite{J,W}).

 One of the principal motivations for this article is to establish this transition law  for all Hamiltonian holomorphic $S^1$ actions, with no conditions on the
fixed point set or on the analyticity of the \kahler metric $\omega$. We obtain the interface asymptotics from the  Gaussian asymptotics of the equivariant 
kernels \eqref{SPECPROJ}.  
The Gaussian asymptotics of   the equivariant
Bergman kernels in Theorems \ref{E} and \ref{EQUIVINT} are used in 
Theorem \ref{thm:interface}  
 to give Erf 
asymptotics for partial Bergman kernels \eqref{RSG}, which are essentially integrals
of equivariant Bergman kernels.
 In Theorem \ref{thm:bulk} we give exponentially decaying asymptotics
of  the partial Bergman kernels $\Pi_{k, P}$ in the forbidden region. 
%

Asymptotics of equivariant \szego kernels have also been studied  by R. Paoletti
 in several settings, of which the closest to this article are contained in   \cite{P, P2}. Equivariant \szego kernels were not explicitly defined or studied by Ross-Singer \cite{RS}; as discussed in \S \ref{OTHERSECT}, they constructed  kernels $G_{j,k}$
 which play the role of $\Pi_{k, j}$.

\ss{Set-up}
\label{setup}

 To state our results we introduce some notation. Let $(L, h, M, \omega)$ be a \Kahler manifold with a positive line bundle $(L, h)$ with $C^{\infty}$ Hermitian metric $h$ and with $\omega = i \ddbar \log h$ a $C^{\infty}$ \kahler form.   Let $H: M \to \R$ be a Hamiltonian function generating the holomorphic $\T$-action. We shift $H$ by a constant such that the minimum of $H$ is zero. 
  In \S \ref{BBSECT} and \S \ref{MORSESECT} the complex
  and real Morse theory of Hamiltonians generating holomorphic $S^1$ actions is reviewed. The Hamiltonian and gradient flows of $H$ commute and generate a $\C^*$ action.  We denote the $\C^*$  action by   $e^{w} z = e^{\rho + i \theta} \cdot z$, and the $\R$-action of gradient flow (and $\T$-action of Hamiltonian flow, resp) of $H$ by $e^{\rho}$ (and $e^{i \theta}$, resp), and 
 infinitesimal generators for $\R$ and $\T$ action by $\pa_\rho$ and $\pa_\theta$.

We assume that $E$ is a regular value of $H$, i.e. $H$ has no critical points on $H^{-1}(E)$. The subinterval $P$ of $H(M)$, is taken as $P=[0, E)$. The allowed region and forbidden region are then 
\[ \acal_E = \{z \in M \mid H(z) < E\}, \quad \text{ and } \fcal_E = \{z  \in M \mid  H(z) > E\}. \]
 Results for general $P$ can be obtained from this easily. 
  
Let $M_{\max}$ be the open dense subset in $M$ containing points where the $\C^*$-action acts freely. Let $M_{\max}^E \subset M_{\max}$ be the set of points whose $\C^*$-orbit (hence $\R$-orbit) intersects with hypersurface $H^{-1}(E)$.  Let $\fcal^E_{\max} = M^E_{\max} \cap \fcal_E$. 

\bd \label{d:zebe}For $z \in M_{\max}^E$, we define $z_E \in H^{-1}(E)$, and real number $\tau_E(z)$ and $b_E(z)$ as follows: \\
(1) $z_E$ is the intersection of the $\R$-orbit $\{e^{\rho} \cdot z\}$ with $H^{-1}(E)$. We define the projection
\[ q_E: M^E_{\max} \to H^{-1}(E), \quad z \mapsto z_E. \]
(2) $\tau_E(z)$ is the `flow-time' from $z_E$ to $z$, $z = e^{\tau_E(z)} \cdot z_E$. \\
(3) $b_E(z)$ is an analog of distanced-squared to $H^{-1}(E)$, defined by
\begin{equation} \label{bE} b_E(z) =  
2 \int_0^{\tau_E(z)}\left(
H(e^{\sigma} \cdot z_E) -H(z_E) \right) \, d\sigma\;. \end{equation} 
For ease of notation, we sometimes write $b(z,E) = b_{E}(z)$, $\tau(z,E) = \tau_{E}(z)$.\\
(4) For $E$ a regular value of $H$, we define the largest $(1/k) \Z$ lattice points in $P=[0,E)$, as
\be  E_k := \max \left\{ \frac{1}{k} \Z \cap [0, E) \right\}.\label{Ek} \ee
\ed
For each point $z \in M$ that is not a fixed point of $\T$, we fix a local $\C^*$-invariant holomorphic section $e_L \in \Gamma(U, L)$ in an open neighborhood $U$ of $z$ and define the \Kahler potential $\varphi$ by $e^{-\varphi} = \|e_L\|_h^2$. For any subspace $\scal_k \subset \Gamma(M, L^k)$,
the Bergman density for $\scal_k$ can be written as
\[ \Pi_{\scal_k} (z) = \sum_{j=1}^{\dim \scal_k} \|s_j(z)\|_{h^k}^2 = \sum_{j=1}^{\dim \scal_k} |f_j(z)|^2 e^{-k\varphi(z)} \]
where $\{s_j: j=1, \cdots, \dim \scal_k\}$ is an orthonormal basis for $\scal_k$ and  $s_j(z) = f_j(z) \cdot e_L(z)^{\ot k}$ for a local holomorphic function $f_j(z)$ on $U$.

\subsection{\label{EQBKintro} Asymptotics of equivariant Bergman kernels }

Our first result is the precise statement that  $k^{-m + 1/2} \Pi_{k, j}(z,z)$ is a kind of Gaussian bump  along
  $H^{-1}(\frac{j}{k})$, i.e. in the tangential directions along $H^{-1}(\frac{j}{k})$ it is essentially constant while it has Gaussian decay at speed $k$ in the normal directions (i.e. along the $\nabla H$ flow lines).  Note that the `lattice points' $\frac{j}{k}$ are Bohr-Sommerfeld type energy levels and $H^{-1}(\frac{j}{k})$ are the corresponding classical energy surfaces. As $k \to \infty$ they become denser and approximate any energy level $E$. We now give
  the Gaussian asymptotics of  $k^{-m + 1/2} \Pi_{k, j}(z,z)$ as $\frac{j}{k} \to E$.

 \begin{Theo} \label{E} Let $(L, h) \to (M, \omega)$ be a positive line bundle over a  \Kahler manifold, $\omega = i \ddbar \log h \in C^{\infty}$,   and let $H: M\to \R$ generate a holomorphic $S^1$-action. Let $E$ be a regular value of $H$, and $z \in M^E_{\max}$.  Then for any sequence $j_1, j_2, \cdots $ such that $|j_k/k - E| < C/k$ for some constant $C$, then the equivariant density
  of states has the following asymptotics.
  
  \[ \Pi_{k,j_k}(z) = \left\{ \begin{array}{ll}  k^{m-\half} \sqrt{\frac{2}{\pi \pa^2_\rho \varphi(z_E)}} (1+O(k^{-1})), \;\;\ & z \in M^E_{\max} \cap H^{-1}(E), \\&\\
k^{m -\half}\sqrt{\frac{2}{\pi \pa^2_\rho \varphi(z_E)}} e^{- k b(z, j_k/k)} (1+O(k^{-1})), & \;\;\ z \in M^E_{\max} \RM H^{-1}(E) \end{array} \right. \] 
 Here,  $z_E$ and $b_E$ are defined in Definition \ref{d:zebe}.

 \end{Theo}
 
\brem \label{rem:jk}
The right side of the bottom asymptotics can be re-written in terms
of $b(z,E)$ but the coefficient changes.
For $z \notin H^{-1}(E)$,  $|b(z, j_k/k) - b(z, E)| = \partial_E b(z, \cdot) (\frac{j_k}{k} - E) + O(\frac{1}{k^2})$. Hence  \[  e^{-k b(z, j_k/k)} \simeq e^{- k b(z,E) }
e^{-  \partial_E b(z, E) [k(\frac{j_k}{k} - E)] }(1 + O(\frac{1}{k})).
\]
Hence, the coefficient of the exponential decay $ e^{-k b(z, E)} $ depends both on the geometric coefficient $  \partial_E b(z, E)$ and on the degree of approximation of $E$ by the nearest `lattice point' $\frac{j_k}{k}$.
\erem
%
%
%

 The next result concerns the scaling asymptotics of the equivariant Bergman
 kernels in a $\frac{C}{\sqrt{k}}$-neighborhood of $H^{-1}(E)$.

  \begin{Theo} \label{EQUIVINT} With $(L,h, M, \omega)$ , $(H, E)$ and $(k,j_k)$ as in Theorem \ref{E}. Let $z_E \in M_{\max} \cap H^{-1}(E)$ and $z_k = e^{ \frac{\beta}{\sqrt{k}} } \cdot z_E$ be a sequence of points approaching $z_E$. Then,
  \[ \Pi_{k, j_k} (e^{ \frac{\beta}{\sqrt{k}}} \cdot z_E)  = k^{m-\half} \sqrt{ \frac{2}{  \pi \pa_\rho^2\varphi(z_E)}}  e^{-  \beta^2  \pa_\rho^2\varphi(z_E)}(1+O(k^{-1/2})). \]
  \end{Theo}

\subsection{Asymptotics of partial Bergman kernels}

In this section, we state analogues of Theorems \ref{E} and \ref{EQUIVINT} for partial Bergman kernels \eqref{PkPdef}; the first is  the analogues of Theorem 1.2 of \cite{ShZ} for partial Bergman kernels of
toric \kahler manifolds. Our aim is to obtain exponentially accurate asymptotics in the forbidden region.

\begin{Theo}\label{thm:bulk} Let  $(L,h, M, \omega)$ and $(H, E)$ be as in Theorem \ref{E}, with 
$h, \omega \in C^{\infty}$.  Let $P = H(M) \cap (-\infty, E)$ and $z \in M^E_{\max}$. 
Then the partial Bergman density is given by the asymptotic
formulas: 
%
\[ 
\Pi_{k, P}(z) = 
\begin{cases} 
\Pi_k(z) + O(k^{-\infty}) & H(z) < E\\
&\\
 k^{m-1/2} \sqrt{\frac{2}{\pi \pa_\rho^2 \varphi (z_E)}} \frac{e^{-kb(z, E_k)}}{1-e^{-|2\tau_E(z)|}}  (1+O(k^{-1}))  &  H(z) >E 
 \end{cases}
 \]
where $z_E, \tau_E(z), b_E(z), E_k$ are given in Definition \ref{d:zebe}, and
 the remainder estimates are uniform on compact
subsets of $M_{{\rm max}}^E$.
\end{Theo}
\brem
As in Remark \ref{rem:jk}, the decaying exponent of $e^{-kb(z, E_k)}$ in the forbidden region (bottom equation in Theorem \ref{thm:bulk}) can be replaced by $e^{-kb(z,E)}$ by changing $E_k$ to $E$ at the cost of introducing an $O(1)$ multiplicative factor, 
\[ e^{-kb(z, E_k)} = e^{-kb(z, E)} e^{ \pa_E b(z,E)  k(E-E_k) }(1+O(1/k))\] However, we assure the reader that when doing computation involving $k^{-1} \log \Pi_{k,P}(z)$ (or $k^{-1} \ddbar \log \Pi_{k,P}(z)$), one may replace $b(z,E_k)$ by $b(z,E)$ with only an $O(1/k)$ error. 
\erem

\subsection{Interface asymptotics}

    Interface asymptotics concerns scaling asymptotics of the partial Bergman kernels for spectral intervals $[E_1, E_2]$   for $z$ 
in a $\frac{1}{\sqrt{k}}$-tube around the interfaces $H^{-1}(E_j)$. It suffices
to consider intervals of the form $(-\infty, E]$ or $[E, \infty)$.  We parametrize the tube around $H^{-1}(E)$  using the $\R_+$-action  as points $z_k = e^{\beta/\sqrt{k}} z_E$ with $z_E \in H^{-1}(E)$.

\begin{Theo} \label{thm:interface}
Let  $(L,h, M, \omega)$ and $(H, E)$ be as in Theorem \ref{E}. In particular, $h, \omega \in C^{\infty}$. Let $z_E \in M_{\max} \cap H^{-1}(E)$ and $z_k = e^{ \frac{\beta}{\sqrt{k}} } \cdot z_E$ approach  $z_E$ along an $\R_+$ orbit for a fixed $\beta \in \R$. Then,
\be \label{PkEzk}
\Pi_{k, (-\infty, E]}(z_k)
 = k^m \Erf\left(\sqrt{4\pi k} \frac{E- H(z_k)}{|\nabla H|(z_E)}\right) (1+O(k^{-1/2})) = k^m \Erf\left( \frac{-\beta |\nabla H(z_E)|}{\sqrt \pi} \right)(1+ O(k^{-1/2})).   \ee
\end{Theo}

We give two  proofs of Theorem \ref{thm:bulk} and  \ref{thm:interface} on partial Bergman kernel $\Pi_{k,P}(z)$.  The first approach (see Section \ref{PENG})  is based on asymptotics of   the equivariant Bergman kernels $\Pi_{k,j}(z)$ for $j/k \in P$
of Theorem \ref{EQUIVINT}.  For fixed $z$, we use the localization Lemma
\ref{lm:localization} to 
 identify the relevant cluster of weights $j$ that contributes to leading order  and calculate their contribution. The second approach (see Section \ref{INTERSECT}) makes use of  the {\it
 polytope character} $\chi_{k, P}(e^w) = \sum_{j \in k P} e^{j w}$ and the Euler-MacLaurin formula as in \cite{ShZ, RS}.  It is a more global approach. Convolution of  $\chi_{k, P}(e^w)$ with the full Bergman kernel sifts out the relevant equivariant modes. A key point is that the $\chi_{k, P}$ is a `semi-classical Fourier integral with complex phase,' as is the Bergman kernel, so that asymptotics can be obtained by use of the Boutet-de-Monvel-\Sjostrand parametrix
(see Section \ref{BSjSECT})  and  the stationary phase method. 

There are interesting variations on the $\frac{1}{\sqrt{k}}$-scaling, which does not 
seem to have been studied before,  and which will be developed in \cite{ZZ17}. In Theorem \ref{thm:interface}, one takes a fixed interval 
of eigenvalues and studies the behavior of pointwise Weyl sums for $z$
in a $\frac{1}{\sqrt{k}}$-tube around the classical interface $H^{-1}(E)$. 
But one might also study spectral sums where the eigenvalues are also
constrained to lie $\frac{1}{\sqrt{k}}$ close to $E$. This is done in Propositions \ref{INTERFACEf} and \ref{EMINTERFACE} for smooth, resp. sharp interval,
 constraints. 
 It is shown in the ``Localization Lemma''  \ref{lm:localization} (2)-(3) that   weights $\frac{j}{k}  \leq E$ in the sum \eqref{PkEzk} which are
`far' from $E$ in the sense that $|\frac{j}{k}- E| \geq \frac{\log k}{\sqrt{k}}$
do not contribute to the  asymptotics and, moreover, that those satisfying
$|\frac{j}{k}- E| \geq C \frac{\log k}{\sqrt{k}}$ do not contribute to the leading order asymptotics for $C$ sufficiently large. Hence a smooth model for the
interface sums is to use weights $ f(\sqrt{k}(\frac{j}{k} - E)) $ with Schwartz
test functions $f$. For $f = {\bf 1}_{[-M, M]}$ we use the Euler-MacLaurin
sums method (Proposition \ref{EMINTERFACE}).  As in Theorem \ref{thm:interface}, the sums in Propositions \ref{INTERFACEf} and \ref{EMINTERFACE} exhibit  Erf asymptotics.  As explained further in \cite{ZZ17}, the   $\sqrt{k}$ scaling of smooth Weyl sums gives rise to a kind of Central Limit Theorem for deterministic Weyl sums.

The asymptotics in Theorem \ref{thm:bulk}  improve on Theorem 1.1 of \cite{RS} by giving   exponentially accurate asymptotics in the forbidden
region and give the analogue of the mass density results of \cite{ShZ}.  The 
interface asymptotics of Theorem \ref{thm:interface} extend the result of \cite{RS} to general holomorphic $S^1$ actions. The method of proof is rather different and, in particular, Proposition \ref{INTERFACEf}   extends to general Hamiltonian flows (in progress; \cite{ZZ17}).

\ss{Zero locus of a Random section}

As in \cite{ShZ}, one may deduce the formula for the asymptotic distribution of zeros of Gaussian random sections of $\scal_k = \scal_{k,P}$, defined in \eqref{SkP}. The definition of random sections in  is precisely as in \cite{ShZ}: Let $s =\sum_{j=1}^{\dim \scal_k} a_{k,j} s_{k,j}$ where $a_{k,j}$ are i.i.d. complex $N(0,1)$ random variables and $\{s_{k,j}\}$ is an orthonormal basis of $\scal_k$. Let $Z_s $ be the zero set of $s$ and let $[Z_s]$ be the current of integration over
$Z_s$.

\begin{Theo} \label{Zth}
For any compact subset $K \subset \acal_E$, resp. $K \subset \fcal_E \cap M^E_{\max}$, we have the following weak* convergence
\[ \lim_{k \to \infty} \frac{1}{k} \E ([Z_s])(z) = \left\{ \begin{array}{ll} \omega, & {\rm for} \;K \subset \acal_E. \\ & \\
  \frac{\sqrt{-1}}{2\pi} \ddbar(\varphi(\wb q_E(z)) +  2 E\tau_E(z)), &{\rm for}\; K \subset  \fcal_E \cap M^E_{\max}. \end{array} \right. \] The limit current in $\fcal_E$  is a smooth $(1,1)$-form of rank $(n-1)$. 
\end{Theo}
Above, $\varphi$ is a local \kahler potential for $\omega$, i.e. $\omega = i \ddbar \phi$. It is $S^1$ invariant and descends to a potential on the reduced \Kahler manifold $H^{-1}(E)/S^1)$.  Also, $q_E$ is defined in Definition \ref{d:zebe}. Since the Gaussian ensemble is $S^1$ invariant, both limit currents are $S^1$ invariant. In $\fcal_E$ the first term is invariant under the $\C^*$ action. When $\tau_E$ is holomorphic, $\ddbar \tau_E = 0$ and the entire limit current is $\C^*$ invariant.

\subsection{Relations to Bernstein polynomials}

The  interface result of Theorem \ref{thm:interface} reduces to a classical theorem on Bernstein polynomial  approximations of characteristic functions of intervals  in the special case where $M = \CP^1$ and $\varphi(z) = \log (1 + |z|^2)$,   resp.   $M = \C$ and $\varphi(z) = |z|^2$. We briefly review these  classical results and their relation to the present article.  

We recall that Bernstein polynomials of one variable give
canonical uniform polynomial approximations to continuous
functions  $f \in C([0,1])$:
\begin{equation} \label{1DBERNDEF} B_N(f)(x) = \sum_{j = 0}^N {N \choose j} f(\frac{j}{N}) x^j
(1 - x)^{N-j}, 
\quad \text{and} \,
\lim_{N \to \infty} B_N(f) \to f \text{ uniformly on $[0,1]$.}
\end{equation}
It is explained in \cite{Ze2,Fe} that \eqref{1DBERNDEF} can be put in the form 
of the kernels in  Theorem \ref{thm:bulk}. More precisely, Bernstein polynomials in the sense of  \cite{Ze2,F} are functions 
 \begin{equation} \label{BERN}    B_k(f)(z): = \sum_{j = 1}^{N_k} f \left(\frac{j}{k}\right) \frac{\Pi_{k,j}(z)}{\Pi_k(z)} = \frac{(\Pi_k   f(\hat H_k)   \Pi_k) (z,z)}{\Pi_k (z,z)}
, \end{equation}  where for $f$ is smooth there exists asymptotic expansion for $k \to \infty$. 
Here,  
$f(\hat{H}_k)$ on $H^0(M, L^k)$ is defined by the spectral theorem,
so that 
$ f(\hat{H}_k) \hat{s}_{k,j} =
f\left(\frac{j}{k}\right) \hat{s}_{k,j}
$ if $s_{k,j}$ is an eigensection of $\hat{H}_k$ with eigenvalue $j/k$. 

 Partial Bergman
kernels are Bernstein polynomials (up to a normalization constant) in the case where $f$ is a step function
${\bf 1}_{P}$. 
Since  ${\bf 1}_P$ above is a characteristic function with a jump,
$B_N({\bf 1}_P)(x)$ cannot approach ${\bf 1}_P(x)$ at the jump. In fact, there is a kind of mean value
formula at the jump involving incomplete Gaussian  integrals $\Erf(x) = \int^x_{-\infty} e^{-x^2/2} dx/\sqrt{2\pi}$. In the classical setting of Bernstein polynomials on $[0,1]$, the jump formula is proved in \cite{Ch,L,Lev}.
The interface asymptotics of \cite{RS} and of this article are generalizations
of Theorem 1.5.2 of \cite{L} in the one-variable setting.


   

 \subsection{Remarks on the proof, on related  work and open problems \label{OTHERSECT} }

The main idea of the present article is to use the spectral theory of the
$S^1$ action and in particular the eigenspace projections (equivariant Bergman kernels) to obtain asymptotics. 
 The spectral viewpoint  generalizes in many respects
to any Hamiltonian $H: M \to \R$ on any compact \kahler manifold, including cases where the
Hamiltonian does not generate an $S^1$ action. In the general case, the gradient
flow and Hamiltonian flow do not commute or define a $\C^*$ action, and the eigenspace projections do not have individual asymptotics. Consequently, 
much of the analysis of this article does not generalize. However it
can be replaced by a more difficult analysis using Toeplitz operators
\cite{ZZ17}. At this time, the spectral approach has been the only feasible approach to asymptotics of  partial Bergman kernels in forbidden regions
or to interface asymptotics.

In this paper we have avoided critical points $\nabla H(z) = 0$ of $H$. The interface results would change at a critical point. Roughly speaking, one would have to use the quadratic approximation (the metaplectic representation) rather than the linear approximation (the Heisenberg representation). It seems to be an interesting problem to study the local interface behavior around a critical point.

In \cite{RS} the role of the equivariant kernels is played by the terms
$G_{n,k} (z,w)\sigma^n(z) \otimes \overline{\sigma(w)}^n$ in Definition 5.21,
where $\sigma$ is defined in Section 5.2 as the section $\sigma \in H^0(Y, \ocal(Y))$ defining a hypersurface component of ${\rm Fix}(\T)$. We do 
not assume ${\rm Fix}(\T)$ contains a hypersurface component and do
not make use of  $\sigma$. We also do not make any constructions of
$G_{n,k}$  or  construct special parametrices adapted to the hypersurface
$Y$, as is done in \cite{RS}.

The analysis in Ross-Singer \cite{RS} was largely motivated by a more general unsolved
problem of determining Bergman kernel asymptotics for subspaces of sections 
 defined in terms of 
    vanishing order along a divisor $Y \subset M$. The partial Bergman kernels
 are   Schwartz kernels of the orthogonal
projections 
\begin{equation} \label{SUB} \Pi_{k}^{Y, t}(z,w): L^2(X, L^k) \to  H^0(X, \ocal(L^k) \otimes \ical_Y^{tk }) \end{equation}
onto the subspace of $s \in H^0(M, L^k)$  which vanish to order $t k$ on a complex
hypersurface $Y$.  
The main question is to find the  asymptotics
 of the density of states,
 \begin{equation} \label{DENSITY} \rho_{h^k}^{Y, t}(z): =   \Pi_{k}^{Y, t}(z,z)_{h^k_z \otimes h^k_z}, \end{equation}
 defined by contracting the \szego kernel along the diagonal with the metric. The  asymptotics depend on whether $z$ lies in the
 {\it allowed region} $\acal_t$ far from the divisor $Y$ or whether it lies in the
 {\it forbidden region} $\fcal_t$ near the divisor $Y$ but it is more difficult to
 define these regions in the absence of a $\T$ action.

  The  general definition of allowed/forbidden regions (due to R. Berman \cite{Ber0} and developed several
 articles of Ross-Witt-Nystrom)  is that the allowed region is the set $\acal_t: =  D_{Y, t} = \{\phi_{e, Y, t} = \phi\}$ where
 a certain equilibrium potential  $\phi_{e, Z, t}$ equals the original \kahler potential. As pointed out in \cite{PS}, 
in this generality, there is no information about the smoothness of $\partial D_{Y,t}$  nor about the
‘transition behaviour’ of \eqref{DENSITY} near $\partial D_{Y,t}$.
In   \cite{Ber0} it is suggested  to employ singular Hermitian metrics with singularities and  with
{\it negative
curvature} 
concentrated along the divisor on $Y$.  At the present time, this program has only partially  been carried out in \cite{Ber0,RS,Co} and remains largely open.
 
The spectral viewpoint towards \eqref{SUB} has not been developed beyond the $S^1$ case of this article, and may not admit generalizations to non-symmetric cases. The interface might be quite irregular in general (as the boundary of an envelope). We briefly discuss how to re-cast vanishing order in terms of spectral theory.
 When the Hamiltonian flow is holomorphic and periodic  and when  one
 component of the fixed point set $M^{\T}$ of the $\T$ action is a divisor $Y$,
 then the allowed and forbidden regions are those defined above in terms
 of the Hamiltonian and the interface asymptotics are given by \eqref{RSG} in  \cite{RS}. The hypersurface $Y$  is necessarily
 the minimum set of the classical Hamiltonian.
 The link between the spectral
  theory and the  definition of partial Bergman kernels in terms of  vanishing order is given in the  following Proposition   (closely related to Lemma 5.4 of \cite{RS}. )

\begin{prop}  \label{VO} Suppose that the minimum set of $H$ is a complex hypersurface $Y$. Then $H^0(X, \ocal(L^k) \otimes \ical_Y^{tk }) = \bigoplus_{j \geq t k} V_k(j)$ is the direct sum of eigenspaces 
of $\hat{H}_k$ for eigenvalue $\geq t$. \end{prop}

\begin{proof} Since $Y$ is $\T$-invariant, both  $L^k$ and $\ical_Y^{tk}$
are $\T$ invariant and so the action fo $\T$ on $H^0(X, \ocal(L^k) \otimes \ical_Y^{tk })$ decomposes into weight spaces,
\[ H^0(X, \ocal(L^k) \otimes \ical_Y^{tk }) = \bigoplus_{j: \frac{j}{k} \in P_0}
V_k(j) \otimes \ical_Y^{tk}. \]
Thus it suffices to determine the summands which are non-zero. An element
$s \in V_k(j)$ transforms by $w^j$ under the action of $w \in \C^*$. We restrict it to $\C^*$ orbit which tends to a point $y \in Y$. In holomorphic coordinates $(w, y)$ where $ w = 0$ on $Y$, it is given by $c_y w^j$. Thus
it vanishes to order $\geq t k$ if and only if $j \geq t k.$
\end{proof}

    Of course, it  is only in special cases that the minimum set of $H$ is a hypersurface. To take
 a simple model example, the hypersurface $\{z_1 = 0\}$ is a component of the fixed point set of the  $S^1$ action on 
$\C^m$ defined by $e^{i \theta} (z_1, \dots, z_m) = (e^{i \theta} z_1, z_2, \dots, z_m)$, but
 for the `isotropic Harmonic oscillator' $e^{i \theta} (z_1, \dots, z_m) = (e^{i \theta} z_1, e^{i \theta} z_2, \dots,  e^{i \theta} z_m)$ generated by $H = \|Z\|^2/2$ only $\{0\}$ is in the fixed point set or minimum set of $H$.

Partial Bergman kernels corresponding to intervals of eigenvalues 
    are closely related to  Bergman kernels
   for the \kahler symplectic cut of $M$ in $H^{-1}(P)$ in the sense of \cite{BGL}.  It would be interesting to compare the interface behavior of the partial Bergman kernel, and the Bergman kernel of the \Kahler cut near the `cut' but only for special cuts does the line bundle project to an ample bundle.

   \subsection{Acknowledgments} We would like to thank R. Paoletti, J. Ross, and M. Singer for their
   comments on earlier versions.

\section{Hermitian line bundles, \Kahler potentials and geometric quantization}

In this section, we review some elementary facts about the \Kahler geometry and geometric quantization, and also establish our notations and conventions.

\subsection{Hermitian line bundles}
Let $(L,h) \to (M, \omega=c_1(L))$ be an ample line bundle with a positive hermitian metric over a projective \Kahler manifold. For any $z \in M$, and any open neighborhood $U$ of  $z$ on which $L$ is trivial,  we may choose a local trivialization $e_L \in \Gamma(U, L)$, that is $e_L(z) \neq 0$ for all $z\in U$. Then we may define the corresponding local \Kahler potential $\varphi: U \to \R$ as 
\[ h(e_L(z), e_L(z)) =: e^{-\varphi(z)}:=h(z). \]
The Chern connection associated to $h$ is  
\[ \nabla: C^\infty(L) \to \acal^1(L) \]
where $C^\infty(L)$ is the sheaf of smooth sections valued in $L$ and $\acal^1(L)$ is the sheaf of smooth $1$-forms valued in $L$, such that 
\[ 
\nabla = \nabla^{(1,0)} + \nabla^{(0,1)}, \quad \nabla^{(0,1)} = \dbar, \quad \nabla^{(1,0)} = \d + A^{(1,0)}, \quad A^{(1,0)} = \d \log h = - \d \varphi.  
\]
The curvature associated with the Chern connection is
\[ F_\nabla(z) = d A^{(1,0)} = d \d \log h = \dbar \d \log h = \ddbar \varphi. \]
We choose the \Kahler form $\omega = Ric(L)$, more precisely
\[ \omega = c_1(L) = \frac{i}{2\pi} F_\nabla = \frac{i}{2\pi} \ddbar \varphi = \frac{-1}{4 \pi} dd^c \varphi. \]
where $d^c = i(\pa - \dbar)$, such that $d^c f  = df \circ J$. 

It often simplifies the analysis to lift sections of $L \to M$ and operators
on sections to the  unit circle bundle $X_h$ of the Hermitian metric $h$,
so that geometric pre-quantization of the $S^1$ action is pullback of scalar
functions under a flow. In the next section we discuss the geometric
aspects of the lift and in \S \ref{XhSECT} we discuss the analytic aspects.

\subsection{The disc bundle and the circle bundle of $L^*$}\label{sec:circle}
Let $(L^*, h^*)$ be the dual bundle to $L$ with the induced hermitian metric $h^*$, which we will also denote as $h$ from now on.  Let $e_L^* \in \Gamma(U, L^*)$ be the dual frame to $e_L$, then we can define the disc and circle bundle in $L^*$: 
\[ D_h = D(L^*): = \{ (z, \lambda) \mid z \in M, \,\lambda \in L^*_z, \|\lambda \|_h \leq 1\}, \qquad X=X_h = \pa D_h. \]
The disc bundle $D_h$ is strictly
pseudoconvex in
$L^*$,  and hence $X_h$ inherits the
structure of
a strictly pseudoconvex CR manifold. Let $\psi$ be a smooth function defined in a neighborhood of $\pa D$ inside $D$, such that $\psi > 0$ in $D^o$ ,  $\psi=0$ on $\pa D$ and $d \psi \neq 0$ near $\pa D$.  For example, one may take   
\begin{equation} \label{rhodef} 
  \psi(x) = - 2 \log | \lambda | + \varphi(z), \end{equation}  
  where $x = (z, \lambda e^*_L(z))$. Associated to $X_h$ is the contact form\footnote{If we used two different defining functions $\psi_1$ and $\psi_2$, the induced $\alpha$s would differ as well. However, if $\psi_1 = f(\psi_2)$, then $\alpha_1 = f'(0) \alpha_2$ only differ by a constant factor.  The two choices of $\psi$ given here differ by a function $f(x) = -\log(1-x)$, with $f'(0)=1$, hence the resulting $\alpha$ is well-defined.}  
\begin{equation} \label{alphadef} \alpha= \Re( i\partial\psi|_X) = \Re(-i\dbar\psi|_X) = d \theta +  \Re (i \dbar \varphi(z)),  \;\;  \pi^* \omega = \frac{1}{2 \pi}  d \alpha\end{equation} 
where we used $(z,\theta)$ to denote $\left( z, e^{i \theta} e^*_L(z) / \|e^*_L(z)\|_h \right) \in X_h$, and we abused notation by omitting $\pi^*$  in $\pi^* \dbar{\varphi(z)}$. The Reeb vector field $R$ is uniquely defined by $\alpha(R)=1, \iota_R d\alpha = 0$;  here it is $R = \pa_\theta$, the fiberwise rotation. Since later we will use $\pa_\theta$ for the generator of the holomorphic circle action on $M$, we will always refer to the Reeb flow by $R$, and the group action by $r_\theta := \exp(\theta R)$. 

A section $s_k$ of $L^k$ determines an equivariant
function
$\hat{s}_k$ on $L^*$ by the rule
$$\hat{s}_k(\lambda) = \left( \lambda^{\otimes k}, s_k(z)
\right)\,,\quad
\la\in L^*_z\,,\ z\in M\,,$$
where $\lambda^{\otimes k} = \lambda \otimes
\cdots\otimes
\lambda$.

\subsection{Lifting the Hamiltonian flow to a contact flow on $X_h$.}\label{LIFT}
Let $H$ be a Hamiltonian function on $(M, \omega)$. Let $\xi_H$ be the Hamiltonian vector field associated to $H$, that is, 
\[ d H(Y) = \omega(\xi_H, Y) \]
for all  vector field $Y$ on $M$. The sign convention for the Hamiltonian vector field and the corresponding Poisson bracket is
\[ df (Y) = \omega(\xi_f, Y), \quad \{f, g\} = -\omega(\xi_f, \xi_g) \]
this choice ensures that $[\xi_f, \xi_g] = \xi_{\{f,g\}}$.

The purpose of this section is to  lift $\xi_H$ to a contact vector field $\hat{\xi}_H$ on $X_h$ and to lift the Hamiltonian $\T$ action to a contact
$\T$ action.  Recall that the  horizontal lift  is defined by ${\xi}_H^h \in \ker \alpha$ and $\pi_*{\xi}_H^h = \xi_H$. We also denote the Reeb vector
field generating the canonical $S^1$ action on $X_h \to M$ by $R$.

\bl Define
\[ \hat{\xi}_H = {\xi}_H^h - 2 \pi H R \]
 Then $\hat{\xi}_H$ is a contact vector field.
\el
\bpf
Since $\pi_* \hat{\xi}_H = \pi_* {\xi}_H^h = \xi_H$, it suffices to check that $\hat{\xi}_H$ preserve the contact form $\alpha$. By \eqref{alphadef}, 
\[\lcal_{\hat{\xi}_H}\alpha = \lcal_{{\xi}_H^h - 2 \pi H R} \alpha = (\iota_{{\xi}_H^h  -  2 \pi H R} \circ d + d \circ \iota_{{\xi}_H^h - 2 \pi H R}) \alpha = \iota_{\xi_H^h} \pi^* (2\pi \omega)  + d(-2 \pi H \alpha(R)) = 0. \]
\epf

In Lemma \ref{PERIODIC} we prove that the lifted flow is periodic of
period $2 \pi$.

\begin{lem} \label{INFGQ} For any $C^{\infty}$ section $s$ of $L^k$,
$\widehat{\hat{H}_k s} = \frac{i}{2\pi k} \hat{\xi}_H (\hat{s}) $. If $H$ defines a holomorphic $S^1$ action, then the spectrum
of $ \hat{H}_k $ is given by
\[ {\rm Sp}(\hat{H}_k) = \{\frac{j}{k} : j \in {\mathbb N},  \frac{j}{k} \in H(M)\}. \]
\end{lem}

\begin{proof}

We write $\xi_H=\xi$.  It is well-known that $\widehat{\nabla_{\xi} s} = 
d \hat{s}(\xi^h)$; we refer to \cite{KN} for the proof. The equation follows from the fact  ${s}$
lifts to  an equivariant function satisfying $R \hat{s} =  i k \hat{s}$.
\end{proof}

Recall the definition \eqref{Ukt} of $U_k(\theta) = e^{i k \theta \hat{H}_k}$, acting on $C^\infty(M, L^k)$, we have
\begin{cor}
\label{lm:twoaction}
For any  smooth section $s \in C^{\infty}(M, L^k)$, and for any $\theta \in \R$, 
\[ \widehat{U_k(\theta) s}  = \exp (-\theta \frac{\hat{\xi}_H}{2\pi} ) \hat{s}=\hat{s} \circ \exp (\theta \frac{\hat{\xi}_H}{2\pi} ). \]
\end{cor}

As mentioned above, if we have a holomorphic $S^1$ action, then the lifted flow is periodic of period $2\pi$ (Lemma \ref{PERIODIC}). It follows that $U_k(\theta)$ is periodic of period $2 \pi$. 

\subsection{\label{2TOR}The $S^1 \times S^1$ action on $X_h$ and its weights}

The Reeb flow and the lifted $\T$ action together define an $S^1 \times S^1$ action on $X_h$.  Its  weights form the semi-lattice   $\{(j,k) \in \Z_+ \times \Z_+, j \in k P_0\}$. This lifting and the approximation of energy levels
by  rays  in $ \Z_+ \times \Z_+$  is discussed in detail in \cite{STZ} for 
toric varieties, and the same discussion applies almost verbatim to $\T$ actions.

The asymptotics of  the  equivariant Bergman 
kernels $\Pi_{k, j}$ involves pairs $(j_n, k_n)$ of lattice points along a ``ray'' in the joint lattice.
The simplest rays are the ``rational rays'' where $j/k \in {\mathbb Q} \cap P_0$. Somewhat more complicated are
``irrational rays'' where $\frac{j_n}{k_n} \to E \notin {\mathbb Q} $. In this
case we consider lattice points with $|\frac{j}{k} - E | \leq \frac{C}{k}$.

\section{\label{T} \Kahler manifolds carrying $\C^*$ actions}

We begin by reviewing the geometry of $\C^*$ actions on \kahler manifolds and
give  examples where at least one component of the  fixed point set is a hypersurface. We also consider
the possible Hamiltonians $H$ which generate such actions. 

\subsection{\label{BBSECT}Bialynicki-Birula decomposition}
Let $(M, \omega)$  be a compact \kahler manifold equipped with a holomorphic $\C^*$ action.  
The generator of the $\C^*$ action 
$\xi \in H^0(M, T^{1,0}) $ is a holomorphic vector field. A holomorphic $\T$ action 
which preserves $\omega$ is necessarily an isometric $\T$ action for the \kahler metric. The closure of a non-trivial $\C^*$ orbit contains two fixed points and is a topological
$S^2$ called a {\it gradient sphere}. A {\it free} gradient sphere is one whose generic 
point has trivial stabilizer. 

By Frankel's theorem \cite{F},  if the action has a fixed point, then the real $S^1$ action is Hamiltonian. We denote the Hamiltonian by $H: M \to \R$
and its Hamilton vector field by $\xi = \xi_H$.  Let $F_1, \cdots, F_r$ be the connected components of the fixed point set
$M^{\T}$. Each $F_j$ is a compact   totally geodesic \kahler  submanifold of $(M, \omega)$. Set
\[ M^+_i:= \{x\in M:  \lim_{t \to 0} tx \in F_i\},\;\;\;M^{-}_{i} :=  \{x\in M:  \lim_{t \to \infty} tx \in F_i\}. \] The so-called {\it Bialynicki-Birula decomposition} \cite{BB,CS}  states that the strata of the  disjoint decomposition  \begin{equation} \label{BB} M= \bigcup M^+_i \; = \; \bigcup M_i^- \end{equation}  are  locally closed analytic submanifolds.
In Theorem II of \cite{CS} it is proved that
\[ T(M_j^{\pm}) |_{F_j} = N(F_j)^{\pm} \oplus T F_j, \]
where $N(F_j)^+$ (resp. $N(F_j)^-$) is the weight space decomposition with positive (resp.
negative) weights. Moreover, there is precisely
one component $M^+_1$ of the plus-decomposition (called the {\it source}),
resp.  one component $M_1^-$ of the minus
decomposition (called the {\it sink}) 
such that the associated strata is Zariski open.  We will denote the maximal stratum $M_1^-$ that flows to the sink by  $M_{{\rm max}}$. To paraphrase
 \cite{BBS},  the $\C^*$ action gives a 'flow' from the source to the
sink, and the  'flowlines' are closures of 'generic' orbits and limits of such closures. 

\subsection{\label{MORSESECT}Morse theory and gradient flow}

The same decomposition can be obtained from the real Morse theory
of the Hamiltonian $H$. 
Kirwan proved that $H^2$  is a minimally degenerate Morse function. Since
we are dealing with a real-valued moment map, we may simply use $H$ and it is also a minimally degenerate (perfect) Bott-Morse function.  The  gradient flow of $H$  with respect to the \kahler metric induces a Morse stratification of $X$, and in  \cite{Ki2,Y} it is proved that this stratification is the same as the Bialynicki-Birula decomposition.
That is, the Morse stratum 
\[ S_{j}^\pm = \{x \in M: \lim_{t \to \pm \infty} \exp ( t \nabla H) \cdot x \in F_{j}\} \]
is the same as $M_j^\pm$.
We note that
\[ M^{\T} = \{x: d H(x) = 0\}\]
so that $F_j$ are the components of the critical point set. The sink corresponds to the minimum set of $H$. In \cite{RS} it is assumed
that one of the components of $M^{\T}$ is a hypersurface, and this
hypersurface is necessarily the minimum set of $H$.

Above we defined the open dense set $M^E_{{\rm max}}$ of points
whose forward or backward gradient trajectories intersect $H^{-1}(E)$. Its complement
consists of the stable/unstable manifolds of critical points other than 
the minimum/maximum. These gradient trajectories can get hung up
at the other critical points and not make it to $H^{-1}(E)$.


We will also identify the Lie algebra $\gfrak$ of $\T$   with $ \R$, such that $ - 2 \pi \pa_\theta \in \gfrak \mapsto 1 \in \Z$. Let $H$ be the corresponding Hamiltonian for $ \xi_H = - 2 \pi \pa_\theta$. $H$ is determined only up to an additive  constant.  We fix the indeterminacy in $H$  by defining $H = 0$ on its minimum set,  the source $M_1^+$. 

  \begin{rem} (Remark on periods) By definition,
the vector field $\partial_{\theta}$ of the action $e^{i \theta} z$ has period
$2 \pi$, so the above convention makes the period of $\xi_H$ equal to $1$.
\end{rem}

\subsection{The image $H(M)$}

We normalized $H$ so that the minimum of $H$ is zero. The question
then arises what is the maximum value of $H$ or equivalently what is
the interval $H(M)$. It must be a ``lattice polytope'', i.e. an interval which
integer endpoints. Thus, the maximum of $H$ must be an integer.

 \begin{lem} \label{HMLEM} Let $z \in M_{{\rm max}}$ and let $\ocal_z  \simeq \CP^1$ be the compactification 
 of $\C^* z$.  Let $\omega(\ocal_z)= \int_{\ocal_z} \omega$. Then 
 \begin{itemize}
 
 \item  $\omega(\ocal_z)$ is a positive integer and is constant in $z$ for $z \in M_{{\rm max}}$. \bigskip
 
 \item $\max H  =  \omega(\ocal_z)$; \bigskip

 \end{itemize}
 
 \end{lem}
 
 \begin{proof} For each $z \in M_{{\rm max}}$ we obtain a polarized \kahler
 $\CP^1$ by $(O_z, L |_{O_z}, \omega |_{O_z})$ and it must be the case
 that $\omega |_{O_z} \in H^2(O_z, \Z)$. This proves the first statement.  
 We then restrict $H : O_z \to \R$. It generates the $S^1$ action restricted to $O_z$.
 Hence $\omega |_{O_z} = (2\pi)^{-1} d H \wedge d \theta$. If $\omega(O_z) =  M$, 
 then 
 $ M = \int_{O_z} (2\pi)^{-1} d H \wedge d \theta =  \int_0^{H_{{\rm max}} }d H =
 H_{\max}$, or $H_{{\rm max}} = M$.
 
 \end{proof}
\subsection{\label{KPSECT}The Hamiltonian and the $\T$-invariant \kahler potentials}
Following \S\ref{EQBKintro}, 
for any $w = e^{\rho + i\theta}  \in \C^*$, denote the $\C^*$ action on $M$ by $z \mapsto e^{\rho + i \theta} \cdot z$. If we choose a local slice $S$ of the $\C^*$ action (necessarily a symplectic manifold),
then we may define {\it slice-orbit} coordinates $(\rho, \theta, y)$ 
by letting $y$ be coordinates on the slice and identifying \begin{equation} \label{SO} e^{i \theta + \rho} y = z. \end{equation} For instance, if we choose a slice $S_E$ of the  $\T$  action on $H^{-1}(E)$ then we may use $S_E \times \C^*$ to give local coordinates
on a neighborhood of $S_E$.  Also, $H^{-1}(E)$ is a slice of the gradient
flow or $\R_+$ action on $M_{{\rm max}}^E$ and we use the coordinates
$(\rho, z_E) \in \R \times H^{-1}(E)$ as well.

As in the introduction, for  $z \in  M_{{\rm max}}^E $, we define  \begin{equation} \label{tauEdef} \tau_E(z) \in \R:= \; {\rm the\; unique\; time\; s.t. }\;
  z= e^{\tau_E(z)}\cdot z_E, \;\;\; z_E \in  H\inv(E). \end{equation}

As in \S \ref{EQBKintro}, we denote the  two global vector fields $\pa_\rho, \pa_\theta$ (not be confused with the Reeb flow $R$ on the circle bundle), such that
\be \label{df:pa_rho}
 \pa_\theta f(z) := \frac{d}{d\theta}|_{\theta=0} f(e^{i \theta} \cdot z), \quad \pa_\rho f(z) := \frac{d}{d\rho}|_{\rho=0} f(e^{\rho} \cdot z) , \qquad \forall f \in C^\infty(M) 
 \ee
Since the $\C^*$ action is holomorphic, we have $\pa_\theta = J \pa_\rho$.

 On any simply connected open set we may define a local $\T$-invariant \kahler potential
 $\varphi$ for $\omega$, so that $\omega = \frac{i}{2\pi} \ddbar \varphi$. We mainly use the
 properties of $\varphi$ restricted to a $\C^*$ orbit $\ocal$. Our choice of coordinates is such that  on $\ocal$, \begin{equation} \label{leafphi} \iota_{\ocal}^* \omega =   \frac{i}{2\pi} \ddbar \iota_{\ocal}^* \varphi = \frac{1}{4\pi} \partial^2_{\rho} \; \varphi \;\; d \rho \wedge d \theta . \end{equation}

Recall that  the gradient vector field $\nabla H
$ is related to  the Hamiltonian vector field $\xi_H$, for any $Y \in Vect(M)$,
\[ dH(Y) = \omega(\xi_H, Y), \quad dH(Y) = g(Y, \nabla H), \quad g(X, Y) = \omega(X, JY) \]
hence $\nabla H = J \xi_H = -2\pi J \pa_\theta = 2 \pi \pa_\rho$. Thus the limit point of downward gradient flow $\nabla H$ is the same as $\lim_{\rho \to -\infty} e^{\rho} \cdot z = z_\infty \in M_1^+$.
   
     The following lemma relates $H$ with the local \Kahler potential. 
\bl \label{lm:ham}
Fix any $z \in M_{max}$. Then 
\[ H(z) = \frac{1}{2} \pa_\rho \varphi(z) \]
\el
\bpf

As in \cite{GS} (5.5) we define a $\T$-invariant potential using  a $\T$-invariant holomorphic section $s \in H_{\T}^0(M, L)$ in the sense of
Lemma \ref{lm:twoaction}, i.e. so that
$\hat{H} s = 0. $
Then,
\[ 0 = \hat{H} s = \frac{i}{2\pi} \nabla_{\xi_H} s +  H s = \frac{i}{2\pi}  \lan A, \xi_H \ran s +  H s = \frac{i}{2\pi} \lan -\pa \varphi, -2\pi \pa_\theta \ran s +   H s \]
where we used the Chern connection 1-form with respect to the basis frame $s$ is given by $A = -\pa \varphi$, and our convention of $\xi_H = -2\pi \pa_\theta$ (see above). Hence 
\be H = - i \lan \pa_\theta, \d \varphi \ran = \lan \pa_\theta, \frac{i}{2}(\dbar - \d) \varphi \ran = \lan \pa_\theta, \frac{-J}{2}(\d + \dbar) \varphi \ran = \lan J \pa_\theta,\frac{-1}{2}(d\varphi) \ran  = \frac{1}{2} \pa_\rho \varphi, \label{eq:HH}
\ee
This definition is unambiguous because
 any two $\T$-invariant holomorphic sections give the same Hamiltonian $\half \pa_\rho \varphi(z)$. Indeed, let $s_1, s_2 \in H^0(M, L)$ be two $\T$-invariant (hence $\C^*$ invariant)  holomorphic sections.  Then $s_1 = f s_2$ for some $\C^*$-invariant meromorphic function $f$.  Then  $\pa_\rho f = \pa_\theta f = 0$, so $\varphi_1(z) = - \log \|s_1\|^2 = - \log|f|^2 + \varphi_2(z)$, and $\pa_\rho\varphi_1(z) = \pa_\rho(- \log|f|^2 + \varphi_2(z)) = \pa_\rho \varphi_2(z)$. 

\epf

\subsection{\label{bphi}The second derivative of $\varphi$ and the action integral $b_E$}

We now consider the relation of $\partial_{\rho}^2 \varphi$ and   $b_E(z)$  \eqref{bE}. Let $\ocal(z)$ denote the $\C^*$ orbit of $z$, and let $\ocal_\R(z)$ denote the gradient trajectory of $z$. If $\ocal_\R(z) \cap H^{-1}(E) \neq \emptyset$, then they intersect at the unique point  $z_E$ 
\eqref{tauEdef}.
\bl \label{bELEM}
If the $\C^*$ orbit of $z$ intersects  $H^{-1}(E)$, let $z = e^{\tau_E(z)} \cdot z_E$ where $\tau_E(z) \in \R$ and $z_E \in H^{-1}(E)$, then 
\begin{equation} \label{bEFORM} b_E(z) = \varphi(z) - \varphi(z_E) - \tau_E(z) \pa_\rho \varphi(z_E). \end{equation}
\[ \pa_\rho b_E(z) = \pa_\rho \varphi(z) - \pa_\rho \varphi(z_E), \quad  \pa_\rho^2 b_E(z) = \pa_\rho^2 \varphi(z) \]
Hence, $b_E(e^\rho \cdot z_E)$ is a strictly convex function in $\rho$, with minimum at $\rho = 0$ and $b_E(z_E)=0$. 
 In particular, for $\rho>0$ (resp. $\rho < 0$), $b_E(e^\rho \cdot z_E)$ is strictly increasing (resp. decreasing) in $\rho$, or equivalently $\tau_E(z)$ along a $\C^* $ orbit  for $z \in \fcal_E$ resp. $z \in \acal_E$.
\el

\begin{proof}
By Lemma \ref{lm:ham}, we have $H(z) = \half \pa_\rho \varphi(z)$, hence $2 E = \pa_\rho \varphi(z_E)$. Hence from \eqref{bE}, we have
\bea
b_E(z) &=& 
-2 E \tau_E(z)+ \int_0^{2 \tau_E(z)}\left[ H(e^{-\sigma/2}\cdot z)\right] \cdot d\sigma
= -2 E \tau_E(z)+ 2 \int_0^{ \tau_E(z)} \left[ H(e^{\sigma}\cdot z_E)\right] \cdot d\sigma \\
& = & -\tau_E(z) \pa_\rho\varphi(z_E) + \int_0^{ \tau_E(z)}\left[ \pa_\rho \varphi(e^{\sigma}\cdot z)\right] \cdot d\sigma 
= \varphi(z) - \varphi(z_E) - \tau_E(z) \pa_\rho \varphi(z_E).
\eea
The other two identities follow from $\pa_\rho \tau_E(z) = 1$ and $\pa_\rho \varphi(z_E) = 0$.

From Lemma \ref{lm:ham} and the fact that $\nabla H = 2\pi \pa_\rho$ (see \S \ref{KPSECT}), we get

\be \label{prho2}  { \pi}\partial_{\rho}^2  \varphi =  |\nabla H|^2 = |\xi_H|^2.  \ee
A closely related formula is that \begin{equation} \label{HE} H(e^\sigma z_0) - H(z_0) = \int_0^{\sigma} \frac{d}{ds} H(e^{s} z_0) ds = \int_0^{\sigma} g( \nabla H(e^s(z_0)) , \frac{d}{ds} (e^s z_0)) ds = \frac{1}{2\pi} \int_0^{\sigma} |\nabla H(e^s(z_0))|^2 ds. \end{equation}
hence for $z = e^{\tau_E(z)} \cdot z_E$ where  $z_E \in H^{-1}(E)$, we have
\[ b_E(z) = \int_0^{\tau_E(z)} \int_0^\sigma \pi^{-1}  |\nabla H|^2 (e^s \cdot z_E) ds d\sigma \]
Monotonicity of $b_E$ in $\tau_E(z)$ is evident from the formula when $\tau_E(z)> 0$ i.e.
$z \in \fcal_E$, hence $b_E$ is monotone increasing in $\rho$.

\end{proof}

\subsection{The Leafwise Symplectic Potential and $b_E$}
In this section we relate $b_E(z)$ to  leaf-wise symplectic potentials. To define the symplectic potentials we use slice-orbit coordinates $(\theta, \rho, y)$ as in \eqref{SO} and   pull back the \Kahler potential \eqref{leafphi} to $\C^*$, by $\varphi(\rho, \theta; y) = \varphi(e^{\rho + i \theta} \cdot y) $. Since the \Kahler potential is $\T$ invariant,  $\varphi(\rho, \theta; y)$ is $\theta$ independent and is convex in $\rho$, and will be denoted by $\varphi(\rho ; y)$ relative to a slice $S_E \subset H^{-1}(E)$. Note that
$\varphi(\rho; e^{i \theta} y) = \varphi(\rho; y)$ and so $\varphi(\rho, z_E)$
is defined for all $z_E \in H^{-1}(E)$.

The leafwise symplectic potential is defined to be the Legendre transformation of $\varphi(\rho; y)$, 
\[ u(I; y) = \sup_{\rho \in \R} (I \rho - \varphi(\rho; y)) \]
Since $\varphi(\rho;y)$ is a smooth convex function in $\rho$, we have
\be \label{LegD1}
 u(I; y) = I \rho(I; y) - \varphi(\rho(I;y);y), \quad \text{ where $\rho(I;y)$ is s.t. } \; I = \pa_\rho \varphi(\rho(I; y) ; y). \ee
The Legendre transformation is an involution,
\be \label{INVOL} \varphi(\rho; y) = \sup_{I \in \R} (\rho I - u(I; y) ) = \rho I (\rho; y) - u(I( \rho;y);y), \quad \text{ where $I(\rho;y)$ is s.t. } \rho = \pa_I u( I(\rho;y);y) \ee
Also their second derivatives are related by
\be \label{LegD2}
\pa_I^2 u(I;y) = \pa_I \rho(E;y) = \frac{1}{\pa_\rho I(\rho;y)} =\frac{1}{\pa_\rho^2 \varphi(\rho;y)} \ee
where $I=I(\rho;y)$ and $\rho=\rho(I;y)$. 
 We use the notation $I$ as in  the ``action-variable' dual to the angle
 variable $\theta$;
 \eqref{LegD1} implies that $I/2$ lies in the range of $H$ (see Lemma \eqref{lm:sympbE}).

The following Lemma relates $b_E(z)$ with the symplectic potential, and can be easily verified using Lemma \ref{bELEM}. 
\bl
Let $z_E \in H^{-1}(E)$ and use gradient flow-coordinates 
$z = e^{\rho } \cdot z_E$. Then
\be \label{b1}  b_{E}(z)  = - u( 2H(z), z_E) - \varphi(z_E) + 2(H(z) - E) \tau_E(z) ,\ee

\el

Using the symplectic potential, one can easily derive the dependence of $b_E(z)$ in terms of $E$ for fixed $z$. 
\bl
\label{lm:sympbE}

Fix $z \in M$ and $E \in H(M)$, such that the $\R^*$ orbit of $z$ intersects $H^{-1}(E)$ at $z_E$, and let $\tau(z,E)= \tau_E(z) \in \R$ be such that $e^{\tau_E(z)} \cdot z_E = z$. Then
$b(z,E) = b_E(z)$ can be written as
\be \label{firstline} b(z, E) = \varphi(z) + u(2E; z)\ee
and we have the following properties
\be \pa_E b(z,E)  = - 2 \tau(z, E), \quad \pa_E^2 b(z,E) = \frac{4}{\pa_\rho^2 \varphi(z_E)}. \label{e:pebe} \ee
Hence $b(z,E)$ is a strictly convex function in $E$ with minimum being $0$ at $E = H(z)$. 
\el

\bpf
First we claim that $\rho(2E; z) = -\tau(z, E)$. Indeed by the definition of $\rho(2E;z)$ in \eqref{LegD1}, we have 
\[ 2E = \pa_\rho \varphi(\rho(2E;z);z) = \pa_\rho \varphi(e^{\rho(2E;z)} \cdot z) = 2 H( e^{\rho(2E;z)} \cdot z)\]
and $\tau(z,E)$ by definition satisfies $E = H(z_E) = H(e^{-\tau(z,E)} \cdot z)$, hence $\rho(2E;z) = -\tau(z,E)$. 
Using \eqref{LegD1} we have
\[ u(2E; z) = 2E \rho(2E;z) - \varphi(\rho(2E;z);z) = 2E (-\tau(z,E)) - \varphi(e^{-\tau(z,E)} \cdot z) = - 2 \tau(z,E) E - \varphi(z_E) \] 
Combined  with  \eqref{bEFORM}, this proves \eqref{firstline}. Next, using $\rho = \pa_I u( I(\rho;y);y)$ from \eqref{INVOL}, we have
\[ \pa_E( b(z,E)) = 2 \pa_I u(I;z)|_{I=2E} = 2 \rho(2E; z) = - 2 \tau(z, E), \]
and using \eqref{LegD2} we have
\[ \pa_E^2 (b(z,E)) = 4\pa_I^2 u(2E;z) = \frac{4}{\pa_\rho^2 \varphi(\rho(2E;z); z)} = \frac{4}{\pa_\rho^2 \varphi(z_E)}.  \]
\epf

\subsection{Periodicity of the lifted flow}

We can now prove the periodicity statement in \S \ref{LIFT}. Recall that the contact vector field is $\hat{\xi}_H = \xi^h_H - 2\pi R$. 

\begin{lem}\label{PERIODIC} The lifted flow $\exp t \hat{\xi}_H$ is $1$-periodic, or equivalently, $U_k(\theta)$ is $2 \pi$-periodic.\end{lem}

\begin{proof}
The equivalence follows from Lemma \ref{lm:twoaction}.  By our choice of generator, the common period of all $\xi_H$-orbits equals $1$, hence flow by $\xi_H^h$ return to the same fiber.  Since on the circle bundle, the horizontal vector field $\xi_H^h$ and vertical Reeb vector field $R$ commute, and  $H(z)$ is constant along the $\xi_H$ orbit, we may first flow by $\xi_H^h$ for time $1$, then by $-2\pi  H R$ for time $1$.
Let $\theta_\gamma$ be defined such that $\exp(\xi_H^h) (z,\lambda) = (z, e^{i \theta_\gamma} \lambda)$. 
Then, 
\[ \theta_\gamma = i \int_\gamma A = i \int_0^{2\pi} \lan -\pa \varphi, \pa_\theta \ran d \theta = 2\pi H \] 
where we used identities from \eqref{eq:HH}. Hence flowing by $-2\pi H R$ sends $(z, \lambda) \mapsto (z, e^{-i 2\pi H} \lambda) = (z, e^{-i \theta_\gamma}\lambda)$, the two $e^{i\theta_\gamma}$ factor cancels, hence $\hat{\xi}_H$ has period $1$.  
  \end{proof}

\subsection{\label{EXSECT} Examples}

To illustrate the variety of $S^1$-\kahler manifolds, we first start with
model linear cases and then proceed to other types of examples.\bigskip

\noindent{\bf (0): Linear actions on $\C^m$} On the non-compact \kahler manifold $\C^m$ with
Euclidean metric the linear $S^1$ actions have the form,
$$e^{i \theta} \cdot (Z_1, \dots, Z_m) = (e^{i b_1} Z_1, \dots, e^{i b_m} Z_m), 
\;\;\; b_j \in \Z, $$
with Hamiltonians $H = \sum_j b_j |Z_j|^2$. Extreme cases are where
all $b_m = 0$ except $b_1 = 1$, in which case the fixed point set is
a hypersurface $Z_1 = 0$, and the isotropic Harmonic oscillator
with all $b_j = 1$ and Hamiltonian $|Z|^2$ with fixed point set $\{0\}$.
\bigskip

\noindent{\bf (i)
Standard $S^1$ actions on $\CP^m$} They arise from subgroups $S^1 \subset SU(m + 1)$ of
the form
$$e^{i \theta} \cdot [Z_0, \dots, Z_m] = [e^{i b_0} Z_0, \dots, e^{i b_m} Z_m], 
\;\;\; b_j \in \Z. $$
With no loss of generality it is assumed that $b_0 = 0$.
When  $b_j \not= b_k \; {\rm when}\; j \not= k$, the action has $m + 1$ isolated fixed points,
$P_j = [0, \dots, 0, z_j, 0, \dots, 0]$. The weights at $P_j$ are $\{b_j - b_i\}_{j \not= i}$.
The Hamiltonian  is 
$$\mu_{\vec b}([Z_0: \cdots : Z_m]) = \frac{b_1 |Z_1|^2 + \cdots + b_m |Z_m|^2}{|Z|^2}. $$
\bigskip

\noindent{\bf (ii) Cubic hypersurface in $\CP^4$}

This example is taken from \cite{Ki2}.
Consider the cubic hypersurface  $X \subset \CP^4$, 
$$x^3 + y^3 + z^3 = u^2  v, $$
in $\CP^4 = \{[x,y,z,u,v]\}$ and let $\C^*$ act on $X$ via the action on $\CP^4,$
$$t \cdot [x, y, z, u, v] = [t^{-1} x, t^{-1} y, t^{-1} z, t^{-3} u, t^3 v]. $$

Then $X^{\T}$ has  three fixed point components, 
$$F_1 = \{[0,0,0,1,0]\},\;\;\; F_2 = \{[x,y,z,0,0]: x^3 + y^3 + z^3 = 0 \}, \;\; F_3 = \{[0,0,0,0,1]\},$$ 
of which two ($F_1, F_3$) are isolated fixed points and $F_2$ is a hypersurface in $X$, i.e. a curve. The point
$P = [0,0,0,0,1]$ is singular.

The corresponding stable sets $S_j = \{x \in X: \lim_{t \to 0} t \cdot x \in F_j\}$ are
$$\left\{ \begin{array}{l} S_1 = \{[x, y, z, u, v] \in X: u \not= 0\}, \\ \\
S_2 = \{[x, y, z, u, v] \in X: u = 0, (x,y,z) \not= 0\}, \\ \\
S_3 = \{[0,0,0,0,1]\}, \end{array} \right. $$
Here, $S_1 $ is Zariski open in $X$, $S_2$ is of codimension one, and $S_3 = F_3$ is a point.

The  Hamiltonian $H: X \to \R$
 is the restriction of the Hamiltonian for the $\T$ action on 
$\CP^m$ above.\bigskip

\noindent{\bf (iii) Ruled surfaces} \cite{HS}

Let $M_g$ be a Riemann surface of genus g, equipped with a constant curvature metric. Let
$L \to M$ be a holomorphic line bundle. $L$ carries a natural $\C^*$ action. Projectivize each line $L_z
\to \P L_z \simeq \CP^1$ to get  $\P L$. It still carries a $\C^*$ action. Examples of $S^1$-invariant \kahler metrics are the constant scalar
curvature metrics.


\section{\label{BSjIBGSECT} The \szego kernel and the Boutet de Monvel-\Sj \, parametrix}
This section is preparation for Theorem \ref{E} and the subsequent asymptotic results. The equivariant Bergman 
kernels $\Pi_{k, j}$ have two positive integer indices, indicating a lattice point in $\Z_+ \times \Z_+$\footnote{The $\T$ action in general has $\Z$ weights, since we have chosen $H$ such that $H(M) \geq 0$, the corresponding weight are in $\Z_+. $}. The asymptotics
in $k$ for a fixed energy level $E$ implicitly involve pairs $(j_n, k_n)$ of lattice points along a ``ray'' in the joint lattice.

It is convenient to lift the sections of $H^0(M, L^k)$, resp.  the equivariant kernels
$\Pi_{k,j}$,  as equivariant functions (resp. kernels) on the unit circle
bundle $X_h \to M$ associated to the Hermitian line bundle $(L^*, h^*)$, see \S \ref{sec:circle}.  This circle bundle carries a canonical $S^1$ action.
The Hamiltonian $\T$ action also lifts to $X_h$ and thus the two commuting circle actions define an $S^1 \times S^1$ action,
whose weights form the semi-lattice  of $\{(j,k) \in \Z_+ \times \Z_+\}$. This lifting and the approximation of energy levels
by  rays  in $ \Z_+ \times \Z_+$  is discussed in detail in \cite{STZ} for 
toric varieties, and the same discussion applies almost verbatim to $\T$ actions. We therefore summarize the key points
and refer to \cite{STZ} for further details.

\subsection{\label{XhSECT} The \szego kernel and the Bergman kernel}
We now discuss the analytic aspects of the lift  to the circle bundle $X_h$ and the disc bundle $D_h$ in   \S \ref{sec:circle} and  \S \ref{LIFT}.
We define the Hardy space
$\hcal^2(X_h)
\subset \lcal^2(X_h)$ of square-integrable CR functions on $X_h$, i.e.,
functions that are annihilated by the
Cauchy-Riemann operator $\dbar_b$
and are
$\lcal^2$ with respect to the inner product
\begin{equation}\label{unitary} \langle  F_1, F_2\rangle
=\int_X
F_1\overline{F_2}dV_X\,,\quad F_1,F_2\in\lcal^2(X)\,.\end{equation}
Equivalently, $\hcal^2(X)$
is the space of boundary values of holomorphic functions on $D$ that
are
in
$\lcal^2(X)$.  Here,  $X_h$ is given the contact  volume
form
\begin{equation}\label{dvx}dV_X= \frac{1}{m!}\frac{\alpha}{2\pi}\wedge
\left(\frac{d\alpha}{2\pi}\right)^m= \frac{\alpha}{2\pi}\wedge dV_M,\quad \text{ where } \quad d V_M = \frac{\omega^m}{m!}.\end{equation}

The $S^1$ action on $X$ commutes
with $\bar{\partial}_b$; hence $\hcal^2(X) = \bigoplus_{k
=0}^{\infty} \hcal^2_k(X)$ where $\hcal^2_k(X) =
\{ F \in \hcal^2(X): F(r_{\theta}x)
= e^{i
k \theta} F(x) \}$. As mentioned in \S \ref{LIFT}, a section $s_k$ of $L^k$ determines an equivariant
function
$\hat{s}_l$ on $L^*$ by the rule
$$\hat{s}_k(\lambda) = \left( \lambda^{\otimes k}, s_k(z)
\right)\,,\quad
\la\in L^*_z\,,\ z\in M\,,$$
where $\lambda^{\otimes k} = \lambda \otimes
\cdots\otimes
\lambda$. We henceforth
restrict
$\hat{s}$ to $X$ and then the equivariance property takes the form
$\hat s_k(r_\theta x) = e^{ik\theta} \hat s_k(x)$. The map $s\mapsto
\hat{s}$ is a unitary equivalence between $H^0(M, L^{ k})$ and
$\hcal^2_k(X)$. (This follows from (\ref{dvx})--(\ref{unitary}) and the fact
that
$\alpha= d\theta$ along the fibers of $\pi:X\to M$.)

We  define the (lifted) \szego kernel $\hat{\Pi}(x,y) $ to be the (Schwarz) kernel of the orthogonal projection
$\hat{\Pi}_k : \lcal^2(X)\rightarrow
\hcal^2(X)$. It is given by
\begin{equation} \hat{\Pi} F(x) = \int_X \hat{\Pi}(x,y) F(y) dV_X (y)\,,
\quad F\in\lcal^2(X)\,.
\label{PiNF}\end{equation}

The Fourier components   $\hPi_k: \lcal^2(X) \to \hcal^2_k(X)$ of the \szego projector can be extracted from $\hPi(x,y)$ by
\begin{equation} \label{FC}  \hPi_k(x,y) = \int_0^{2\pi} e^{-i k \theta} \hPi(r_{\theta} x, y) \frac{d\theta}{2\pi} \end{equation}
The \szego (or Bergman)\footnote{In the setting of line bundles, we use the terms interchangeably.} kernel $\Pi_k(z,w)$ for the orthogonal projection $\Pi_k: \lcal^2(M, L^k) \to H^0(M, L^k)$ can be obtained via the isometry of $H^0(M, L^k) \cong \hcal^2_k(X)$. 


In a local coordinate patch $U$ with a holomorphic frame $e_L \in \Gamma(U, L)$, we introduce two scalar kernels $K_k(z,w)$ and $B_k(z,w)$, with respect to the holomorphic frame and unitary frame: 
\[
\Pi_k(z,w) =: K_k(z,w) \cdot {e}_L^k(z) \ot \overline{{e}_L^k(w)} 
=: B_k(z,w) \cdot \frac{e_L^k(z)}{\|e^k_L(z)\|_h} \ot \overline{\frac{e_L^k(w)}{\|e^k_L(w)\|_h}} 
\]
The Bergman density function $\Pi_k(z)$  is the contraction of $\Pi_{k}(z,w)$ with the hermitian metric on the diagonal, 
\[\Pi_k(z): = B_k(z,z) (:= \Pi_{k}(z,z)), \]
where in the second equality we record a standard abuse of notation in which the diagonal of the
\szego kernel is identified with its contraction.

\subsection{Equivariant \szego kernels}
%

Let  $e_L$ is a local $\T$-invariant holomorphic frame and  we define equivariant Bergman kernel and densities, 
\be
\label{defeqK}\left\{ \begin{array}{l}
\Pi_{k,j} (z,w) = K_{k,j}(z,w) \cdot {e}^k_L(z) \ot \overline{{e}^k_L(w)} = B_{k,j}(z,w) \cdot \frac{e_L^k(z)}{\|e^k_L(z)\|_h} \ot \overline{\frac{e_L^k(w)}{\|e^k_L(w)\|_h}} , \\ \\   \Pi_{k,j}(z) = B_{k,j}(z,z). \end{array} \right. 
\ee

  Equivariant Bergman kernels are closely related to Bergman
kernels for the reductions of the level sets $H^{-1}(\frac{j}{k})$. 
For instance,  the space of  invariant sections
  \begin{equation} \label{INVAR} V_k(0): =  H^0_{\T}(M, L^k) = \{s \in H^0(M, L^k): e^{i \theta} s = s\}. \end{equation}
 is isomorphic in a canonical way to the space of holomorphic
  sections of the reduced line bundle $L_{\T}$ on the reduced space
  $M_E :  = H^{-1}(E)/S^1$, i.e.
  $V_k(0)  \simeq H^0(M_E, L_{\T}^k\}$, so that
   $\dim H^0_{\T}(M, L^k) = {\rm Vol}(H^{-1}(E)/S^1)\; k^{m-1} $.

\subsection{\label{BSjSECT} The Boutet de Monvel-\Sj  parametrix}
Near the diagonal in $X_h \times X_h$, there exists a parametrix due to  Boutet de Monvel-\Sj 
\cite{BSj} for the \szego kernel of the form,  \begin{equation} \label{SZEGOPIintroa}  \hat{\Pi}(x,y) =  \int_{\R^+} e^{ \sigma {\psi}(x,y)} \chi(z_x, z_y) s(x, y ,\sigma) d \sigma  + \hat{R}(x,y). \end{equation} Here,  $\chi(z_x, z_y) $ is a smooth cutoff supported
in a neighborhood of the diagonal of $M \times M$. $\psi(x,y)$ is defined as (up to $2\pi \Z i$ ambiguity)
\[ \psi(x,y) = -\log \lambda_x - \log \wb{\lambda_y} + \varphi(z_x, z_y) \]
where $x =  \lambda_x \cdot e_L^*(z_x) \in X_h$ for $z_x \in M, \lambda_x \in \C^*$,  similarly for $y$, with respect to  a local trivialization $e_L \in \Gamma(U, L)$. And $\varphi(z,w)$ is the almost analytic extension of $\varphi(z)$ (we abuse notation), that is
\[ \varphi(z,z) = \varphi(z), \dbar_z \varphi(z,w) = \d_w \varphi(z,w) = 0 \text{ to infinity order on } \Delta_M \subset M \times M. \]
On the co-circle bundle, we have $2\Re \log \lambda_x = \varphi(z_x)$ and $2\Re \log \lambda_y = \varphi(z_y)$, hence if we write $\theta_x = \arg \lambda_x, \theta_y = \arg \lambda_y$, we have 
\[ \psi(x,y) = -i \theta_x + i \theta_y + \varphi(z_x, z_y) - \varphi(z_x)/2 - \varphi(z_y)/2. \]

The real part of $\psi$ proportional to the Calabi-Diastasis, 
\[ \Re \psi(x,y) = -\half D(z_x,z_y), \]
 where
\begin{equation} \label{CD} 
D(z,w) : =  - \varphi(w,z) - \varphi(z,w) + \varphi(z) + \varphi(w), 
\end{equation}
is defined near the diagonal of $M \times M$, and is positive and only vanishes when $z=w$. 
The amplitude is a classical symbol,     \begin{equation} \label{s} s \sim \sum_{n
= 0}^{\infty} \sigma^{m - n} s_n(x,y). \end{equation}  Finally, the remainder term $\hat{R}(x,y)$ is $C^{\infty}$.

From the parametrix for $\hPi$ one can derive  semi-classical parametrices for the Fourier
components  and thus  for the semi-classical \szego kernels  on 
$H^0(M, L^k)$. 
If we substitute the first term of  \eqref{SZEGOPIintroa} into \eqref{FC}, one obtains
the oscillatory integral,
 \begin{equation} \label{SZEGOPIintro}  \hat{\Pi}_{k}(x,y) \sim  \int_{\R^+} 
\int_0^{2 \pi} e^{ \sigma \psi(x, r_{\theta} y)} e^{i k \theta} \chi(z_x, z_y) s(x, r_{\theta} y ,\sigma) d \theta d \sigma, \end{equation} 
Changing variables $\sigma \to k \sigma$ and eliminating  the $d \theta d \sigma$ integral by the stationary phase method gives,
 at least formally, the off-diagonal
expansion for the full \szego kernel on $M$,
\begin{equation} \label{BBSjeq} \quad K_k(z,w) = e^{k \varphi (z, {w})} k^m \;(1 + O(k^{-1}) ). \end{equation}
A direct construction of the parametrix is  given in \cite{BBSj} (where \eqref{BBSjeq} is stated as (2.2)).

We only use the full Bergman kernel and the the parametrices \eqref{BBSjeq}-\eqref{SZEGOPIintro} in Sections \ref{THESECT} and \ref{FIRSTBSjAPP},  in the proof of the Localization Lemma, and in Section 
\ref{INTERSECT}.

\section{Equivariant Bergman kernels: Proofs of Theorem \ref{E} and Theorem \ref{EQUIVINT}}
In this section, we prove that the equivariant Bergman kernel $\Pi_{k,j}(z,z)$  forms a Gaussian bump around the hypersurface $H^{-1}(j/k)$, with decay width $\sim 1/\sqrt{k}$.

\bl
\label{lm:scalinglaw}
For all $\alpha, \beta \in \C$, we have 
\[ K_{k,j} (e^\alpha \cdot z,e^\beta \cdot  w)  = e^{j(\alpha + \bar{\beta})} K_{k,j}(z,w) \]
and
\[ B_{k,j}(e^\alpha \cdot z,e^\beta \cdot  w)  = e^{j(\alpha + \bar{\beta}) - k(\varphi(e^\alpha \cdot z)-\varphi(z))/2 - k(\varphi(e^\beta \cdot w)-\varphi(w))/2 } B_{k,j}(z,w). \]
In particular, if we set $\beta = -\bar \alpha$, we have
\[ K_{k,j} (z,w) = K_{k,j}(e^{\alpha} \cdot z, e^{-\bar \alpha} \cdot w) . \]
\el

\bpf
This is immediate from the definition of $K$ and $B$. 
\epf

\subsection{\label{THESECT} Proof of Theorem \ref{E}}

Theorem \ref{E} follows from the following two propositions. The first one establishes the decay property of $\Pi_{k,j}(z)$ away from the real hypersurface $H^{-1}(j/k)$. Recall the definition of $b_E(z)$ \eqref{bE}.
\bp
\label{pp:eqv-decay} 
Fix  $(k,j)$ and $z \in H^{-1}(E)$ where $E=j/k$. Then, for any $\alpha \in \R$, we have
\[  \Pi_{k, j} (e^{\alpha} \cdot z) = e^{- k b_E(e^\alpha z)}  \Pi_{k, j}(z)\]
\ep

\bpf
Using Lemma \ref{lm:scalinglaw} with $z=w$, $\alpha=\beta \in \R$, we have
\[ B_{k,j}(e^\alpha \cdot z,e^\alpha \cdot  z)  = e^{2 j \alpha - k(\varphi(e^\alpha \cdot z)-\varphi(z)) } B_{k,j}(z,z) \]
Now, write $j = k E = k H(z) = k \pa_\rho \varphi/2$, we have
\[ \Pi_{k,j}(e^\alpha \cdot z) = B_{k,j}(e^\alpha \cdot z,e^\alpha \cdot  z) = e^{-k (\varphi(e^\alpha \cdot z)-\varphi(z) - \alpha \pa_\rho \varphi(z))} B_{k,j}(z,z) = e^{-k b_E(z)}B_{k,j}(z,z) = e^{- k b_E(e^\alpha z)}  \Pi_{k, j}(z) \]  
\epf

 These are exact identities and do not involve parametrices. Next, we  express $K_{k,j}$ as a Fourier coefficient of $K_k$ with
respect to the Hamiltonian $S^1$ action and give a parametrix formula,

\begin{lem} \label{KkjBSj} For any $j: |j| \leq k$,
\bea 
K_{k,j}(z,z) & = & \int_0^{2\pi}  K_k(e^{i \theta}\cdot z,z) e^{-i j \theta} \frac{d \theta}{2 \pi} \\
&=& k^m\int_0^{2\pi} e^{k\varphi(e^{i\theta} z,z)}  e^{-i j \theta} \chi(e^{i\theta} z,z)(1 + O(1/k))\frac{d \theta}{2 \pi} 
\eea
where $\chi(z,w)$ is a cut-off function supported in a neighborhood $U$ of the diagonal of $M \times M$. 
\end{lem}

\bpf The first line is evident and the second uses \eqref{BBSjeq}.
\epf

The next proposition studies the kernel $\Pi_{k,j_k}(z)$ when $z \in H^{-1}(E)$ and $j_k/k \to E$.  

\bp \label{pp:eqv-onshell}
Fix a regular value $E$ of $H: M \to \R$,  and a sequence $\{j_k\}$ such that $|\frac{j_k}{k} - E| < C/k$ for some positive constant $C$.  Then for any $z \in H^{-1}(E)$ with trivial stabilizer in the $\T$-action,  we have
\[ \Pi_{k,j_k}(z) =k^{m-1/2} \sqrt{\frac{2}{\pi \pa^2_\rho \varphi(z)}}(1 + O(1/k)).  \]
\ep
\bpf
Let $E_k = j_k/k$, and $z_k \in H^{-1}(E_k)$ $\rho_k \in \R$, such that $z = e^{\rho_k} z_k$. We have $|\rho_k|=O(1/k)$. Indeed, 
\[ C/k > |E_k - E| = \frac{1}{2} |\pa_\rho \varphi(z_k) - \pa_\rho \varphi(z)| =  \frac{1}{2}  |\int_0^{\rho_k} \pa_\rho^2 \varphi(e^s z) d s| > C'  |\rho_k| \] 
where we used the fact $\varphi$ is psh and $\T$-invariant, to get $\pa_\rho^2 \varphi$ strictly positive, hence $|\rho_k|=O(1/k)$.  Then using Proposition \ref{pp:eqv-decay}, we get
\be \Pi_{k,j_k}(z) = \Pi_{k,j_k}(z_k) e^{-k b_{E_k}(e^{\rho_k} z_k)} = \Pi_{k,j_k}(z_k) e^{-k O(\rho_k^2)}= \Pi_{k,j_k}(z_k) (1+O(1/k)) \label{e:znzk} \ee
Next, we evaluate $\Pi_{k,j_k}(z_k)$ using the parametrix of
Lemma \ref{KkjBSj}  and the stationary phase method.

Setting $j = j_k, z= z_k$ in Lemma \ref{KkjBSj}, we get
\begin{equation}  K_{k, j_k} (z_k, z_k) =  k^m \int_{-\pi}^{\pi} e^{  k ( \varphi(e^{i\theta} z_k,z_k)- i E_k \theta) }\chi(e^{i\theta} z_k,z_k)(1 + O(1/k))\frac{d \theta}{2 \pi}. \end{equation}
This is not quite a standard stationary phase integral because the phase 
\begin{equation} \label{psik} \Psi_k(i \theta): =  \varphi(e^{i\theta} z_k,z_k)- i E_k \theta,  \end{equation} depends on $k$. However, all aspects of the stationary 
phase expansion (see \cite{Ho}) extend with no essential change to  \eqref{psik}.

We claim that $\theta=0$ is a Morse critical point of \eqref{psik}. To see
this, we use $ H = - i \lan \pa_\theta, \d \varphi \ran$ from  \eqref{eq:HH}'s first equality. Thus, the first derivative of $\Psi_k(i\theta)$ at $\theta=0$ is
$$-i\partial_{\theta} \Psi_k(i \theta) |_{\theta = 0} = H(z_k) - E_k = 0. $$
To calculate the second derivative at $\theta = 0$
we rewrite
$ \varphi(e^{i \theta} z, z) =  \varphi(e^{i \theta/2} z, e^{-i \theta/2} z) $, using the $\T$ invariance of $\varphi$, then we
extend \eqref{psik} to a holomorphic function
\begin{equation} \label{PSIDEF} \Psi_k(\tau) = \varphi( e^{\tau/2} \cdot z_k, e^{\bar{\tau}/2} \cdot z_k) -  E_k \tau \end{equation}
The Taylor expansion of $\Psi_k |_{\R}$, has the form,
\[ \Psi_k(t) = \varphi(e^{t/2} z_k, e^{t/2} z_k) - t E_k = \varphi(e^{t/2} z_k) - t E_k = \varphi(z_k) + \frac{t^2}{8} \pa_\rho^2\varphi(z_k) + O(t^3). \]
Thus
 $\theta=0$ is a non-degenerate isolated critical point of $\Psi(i\theta)$, with Hessian $\pa_\theta^2|_{\theta=0}\Psi(i\theta) = - \Psi''(0) = -\frac{1}{4} \pa_\rho^2\varphi(z)$. Hence, we may choose $\epsilon>0$ small enough such that for $|\theta|<\epsilon$ there is no other critical point than $\theta=0$. Let $\eta(\theta) \in C^\infty_c(\R)$, such that $\eta(\theta) \equiv 1$ for $|\theta|<\epsilon/2$ and $\eta(\theta) \equiv 0$ for $|\theta|>\epsilon$. Since $e^{i \theta} \cdot z \neq z$ only when $\theta \neq 0$ by the assumption on the stabilizer of $S^1$-action on $z$, and since (as in \eqref{CD})
\[ \Re \varphi(z,w) - \half \varphi(z) - \half \varphi(w) = - \half D(z,w) \leq 0 \]
and only vanishes when $z=w$, we have the following upper bound on the real part of the phase function
\be \sup \{\Re \varphi(e^{i\theta} \cdot z, z) - \varphi(z) \mid (e^{i\theta} \cdot z, z) \in U, \, |\theta| \in [\epsilon/2, \pi] \} = -c < 0. \label{e:repsi} \ee
for some positive constant $c$, where we used $\varphi(e^{i \theta} z) = \varphi(z)$ by the $\T$-invariance of $\varphi$.  It follows that
\bee 
\Pi_{k,j_k}(z_k) &=& e^{- k \varphi(z_k)} K_{k,j_k}(z_k,z_k) \notag\\
& =&  k^m \int_{-\pi}^{\pi} e^{  k ( \varphi(e^{i\theta} z_k,z_k) - \varphi(z_k)- i E_k \theta) }\eta(\theta)\chi(e^{i\theta} z_k,z_k)(1 + O(1/k))\frac{d \theta}{2 \pi} \notag\\
&+& k^m \int_{-\pi}^{\pi} e^{  k ( \varphi(e^{i\theta} z_k,z_k) - \varphi(z_k)- i E_k \theta) }(1-\eta(\theta))\chi(e^{i\theta} z_k,z_k)(1 + O(1/k))\frac{d \theta}{2 \pi}  \notag \\
&=&   k^{m-1/2} \sqrt{\frac{2}{\pi \pa^2_\rho \varphi(z_k)}} (1 + O(1/k)) \label{e:pikzk}\\
& =& k^{m-1/2} \sqrt{\frac{2}{\pi \pa^2_\rho \varphi(z)}} (1 + O(1/k)) \notag
\eee
where we applied stationary phase method to the first term and bound the second term by  $O(e^{-c k })$ using \eqref{e:repsi}. Since   $z = e^{\rho_k} z_k$ with  $|\rho_k|=O(1/k)$, we  replaced $z_k$ by $z$ in the last step without changing the remainder estimate.  Combining  \eqref{e:znzk} and \eqref{e:pikzk}, we finish the proof of the proposition. 
\epf

\brem
If the stabilizer $G_z$ of  $z$ is non-trivial then it is a cyclic group generated
by  $\zeta =  e^{ \frac{2\pi i}{\ell} }$ for some positive integer $\ell$. By Lemma \ref{lm:scalinglaw} and by the stabilizer condition, 
$ K_{k,j_k} (e^{ \frac{2\pi i}{\ell} } \cdot z, z)  = e^{\frac{2 \pi i j_k }{\ell}} K_{k,j_k}(z,z) = K_{k, j_k} (z, z)  .$ Hence, either $ K_{k, j_k} (z, z) = 0$ or
$ e^{\frac{2 \pi i j_k }{\ell}} =1$, i.e. $\frac{j_k}{\ell} \in \Z$.

The stationary phase method applies as well, and each element
$\zeta^n, n = 0, \dots, \ell-1$ is a critical point of the $d \theta$ integral. The phase has the critical value $ e^{ \frac{2\pi i j_k n  }{\ell} }$ at $\zeta^n$.   The Hessian is independent of $n$, so the leading term of the stationary phase expansion is 
$$ k^{m-1/2} \sqrt{\frac{2}{\pi \pa^2_\rho \varphi(z)}}  \sum_{n = 0}^{\ell -1} 
e^{ \frac{2\pi i j_k n  }{\ell} }. $$
If $ \frac{j_k}{\ell} \in \Z$ then each term is $1$ and the sum is $\ell$. Otherwise, $ K_{k, j_k} (z, z)  = 0$.
\erem

The above two propositions finish the proof of Theorem \ref{E}.

\begin{rem}
If $\varphi(z)$ is real analytic, then $\Psi(\tau)$ is holomorphic when $\Im(\tau)$ is small enough. If $\varphi$ is only smooth, then $\Psi(\tau)$ is an almost analytic extension of $\Psi|_\R$.   Although the proof uses the parametrix, it only uses $\Psi$ in the real domain and only uses the $C^{\infty}$ remainder. Hence, it does not
require real analyticity.
\end{rem}

\subsection{Proof of Theorem \ref{EQUIVINT}}
 Theorem \ref{EQUIVINT} follows from the following proposition. 
\bp
\label{pp:eqv-local} 
For any fixed $k,j$,  $z \in H^{-1}(j/k)$ and $\alpha \in \R$, we have
\[  \Pi_{k, j} (e^{\alpha/\sqrt{k}} \cdot z) = e^{- \frac{\alpha^2}{2} \pa_\rho^2 \varphi(z) }  \Pi_{k, j}(z)(1+O(k^{-1/2}))\]
\ep
\bpf
This follows from Proposition \ref{pp:eqv-decay}, and Lemma \ref{bELEM}. We Taylor expand $b_E(e^\alpha \cdot z)$ in $\alpha$ around $\alpha=0$, to get
\be
\label{eq:bE2a} 
b_E( e^\alpha \cdot z) = \frac{\alpha^2}{2} \pa_\rho^2 \varphi(z) + g_3(z, \alpha),  \text{ where } g_3(z,\alpha) = O( |\alpha|^3).
\ee
Then we plug in the expansion to the exponent $e^{-k b_E(e^\alpha \cdot z)}$ to get
\bea
\Pi_{k, j} (e^{\alpha/\sqrt{k}} \cdot z) &=& e^{-k b_E(e^{\alpha/\sqrt{k}} \cdot z)}  \Pi_{k, j}(z)\\
&=& e^{-k(\frac{\alpha^2}{2k} \pa_\rho^2 \varphi(z) + g_3(z, \frac{\alpha}{\sqrt{k}}))} \Pi_{k, j}(z) \\& =&e^{- \frac{\alpha^2}{2} \pa_\rho^2 \varphi(z)}  \Pi_{k, j}(z)(1+ O(k  g_3(z, \frac{\alpha}{\sqrt{k}}))) \\
&=& e^{- \frac{\alpha^2}{2} \pa_\rho^2 \varphi(z)}  \Pi_{k, j}(z)(1+O(k^{-1/2})).
\eea
\epf

\section{\label{LOCALSECT} Lemma for Localization of sums }
In this section we consider the sums in
the partial Bergman kernels \eqref{PkPdef}.
We prove several localization formulae for these sums.  Roughly speaking a localization
formula says that, for a  given $z$,
only terms in the sums with $|\frac{j}{k} - H(z) | < \frac{M}{\sqrt{k}}$
contribute to the leading order asymptotics.

\bl \label{lm:localization}
As in Theorem \ref{E}, let $(L, h, M, \omega)$ be a \Kahler manifold with a positive line bundle, and $H$ generates a holomorphic $S^1$-action on $(L, M)$, with $E$ a regular value of $H$, and $z \in H^{-1}(E)$ with $\C^*$ acting freely on $z$.  
Fix any smooth cut-off function $\chi: \R \to [0,1]$, such that $\chi(x) = 1$ for $|x| \leq 1$ and $\chi(x)=0$ for $|x|\geq 2$. Then we have: \\
(1) For any $1/2 \gg \epsilon > 0$,  we have
\be \label{eq:localization}
\sum_{j/k \in H(M)} \left(1 - \chi\left( \frac{|j/k-H(z)|}{k^{-1/2+\epsilon}}\right) \right) \Pi_{k,j}(z) = O(k^{-\infty}).
\ee
(2) For any $R>0$,  there exists $C$ large enough such that
\be \label{LOCALIZATION-log} 
\sum_{j/k \in H(M)} \left(1 - \chi\left( \frac{|j/k-H(z)|}{C k^{-1/2}\sqrt{\log k}}\right) \right) \Pi_{k,j}(z) = O(k^{-R}).
\ee
(3) For any $\epsilon>0$,  there exists $C$ large enough such that for large enough $k$
\be
\label{LOCALIZATION-log} 
 \sum_{j/k \in H(M)} \left(1 - \chi\left( \frac{|j/k-H(z)|}{C k^{-1/2}}\right) \right) \Pi_{k,j}(z) < \epsilon k^m
\ee
The above statements are also true for $\chi(x) = 1_{[0,1]}(x)$. 
\el

\bpf

First we prove (1) and (2). 
From Proposition \ref{pp:eqv-decay}, we have
\[  \Pi_{k,j}(z)  
=  e^{-k b(z,j/k)} \Pi_{k,j}(z_j) =O(k^{m-1/2} e^{-k b(z,j/k)}). \]
 If $j/k > H(z)$ and $1 - \chi\left( \frac{|j/k-H(z)|}{ k^{-1/2} S} \right)$ is nonzero, e.g. for $S = k^\epsilon$ or $C\sqrt{\log k}$, then by monotonicity of $b(z,E)$ in $E$ for $E > H(z)$ (Lemma \ref{lm:sympbE}), we have
\[ kb(z, j/k) > kb(z, H(z) + k^{-1/2} S) = \frac{1}{2} \pa_E^2 b(z,H(z)) S^2 + O(k^{-1/2} S^3). \]
Similar statement is true for $j/k<H(z)$. Hence
\[  \sum_{j/k \in H(M), j/k>H(z)} \left(1 - \chi\left( \frac{|j/k-H(z)|}{k^{-1/2}S}\right) \right) e^{-k b(z,j/k)} \Pi_{k,j}(z_j)  =  O(  e^{-\frac{1}{2} \pa_E^2 b(z,H(z)) S^2 + (m+1/2) \log k}) . \]
If $S = k^\epsilon$, then $-\frac{1}{2} \pa_E^2 b(z,H(z)) S^2 + (m+1/2) \log k < -c k^{2\epsilon}$ as $k \to \infty$, proving \eqref{eq:localization}. If $S = C\sqrt{\log k}$, then for any $R>0$, we can choose $C$ large enough such that 
\[ -\frac{1}{2} \pa_E^2 b(z,H(z)) S^2 + (m+1/2) \log k  = (m+1/2 - C^2 (\frac{1}{2} \pa_E^2 b(z,H(z)))) \log k < - R \log k, \]
proving \eqref{LOCALIZATION-log}. 

To prove (3), it is not enough to have a uniform bound on the summand, one need to prove that the summand decays fast. Consider the range of $j$, where 
\[ I_{k, H(z)} = \{j : k^{-1/2} C <  |j/k -H(z)| < k^{\epsilon-1/2}\} ,\] 
then 
\bea b(z,j/k) &=& b(z,H(z)) + \pa_E b(z, H(z)) (j/k-H(z)) + \half \pa^2_E b(z, H(z)) (j/k-H(z))^2 + O(|j/k-H(z)|^3) \notag \\
&=&  \half \pa^2_E b(z, H(z)) (j/k-H(z))^2 + O(k^{3\epsilon - 3/2})
\eea
hence sum over $j \in I_{k, H(z)}$ gives
\bea \sum_{j \in I_{k,H(z)}} \Pi_{k,j}(z) &=& \sum_{j \in I_{k,H(z)}} \Pi_{k,j}(z_{j}) e^{-k b(z, j/k)} \\
&=&\sum_{j \in I_{k,H(z)}} k^{m-1/2} \sqrt{\frac{2}{\pi \pa_\rho^2 \varphi(z_j)}} e^{-\pa^2_E b(z, H(z)) (\sqrt{k}(j/k-H(z)))^2} (1+O(k^{3\epsilon-1/2})) \\
&<& c_1 k^{m-1/2} \sqrt{k} \int_C^\infty e^{-\pa^2_E b(z, H(z)) x^2} dx \\
&=& c_1 k^m \delta_C, 
\eea
where $\delta_C$ is a constant depending on $C$ and $\pa^2_E b(z, H(z))$, such that as $C \to \infty$, $\delta_C \to 0$. For any given $\epsilon>0$, we may take $C$ large enough, such that $c_1 \delta_C < \epsilon$. Thus combining with part (1) of the proposition, we finished the proof of part (3). 
\epf

\section{\label{PENG}Proof of Theorem \ref{thm:bulk} and \ref{thm:interface}: Summing  Equivariant Bergman Kernels}

In this section we prove results   about partial Bergman kernel asymptotics in the interior of the allowed/forbidden regions (Theorem \ref{thm:bulk} ) and near the interface(\ref{thm:interface}),  using localization Lemma \ref{lm:localization} and asymptotics
of equivariant Bergman kernels in Theorem \ref{E} and Theorem \ref{EQUIVINT}. Let $P = [0, E) \subset H(M)$, and recall the partial Bergman density as $\Pi_{k, P}(z) = \sum_{j/k \in P} \Pi_{k,j} (z)$.  

We fix a standard smooth cut-off function $\chi: \R \to [0,1]$, such that $\chi(x) = 1$ for $|x| \leq 1$ and $\chi(x)=0$ for $|x|\geq 2$. 

\bpf[Proof of Theorem \ref{thm:bulk} ]
{\em (Allowed Region).} If $z$ is in the allowed region, we may use the localization formula for the sum to write
\[ \Pi_{k,P}(z) = \sum_{j \in k P \cap \Z} \chi\left( \frac{|H(z) - j/k|}{k^{-1/2+\delta}} \right) \Pi_{k,j} (z) + O(k^{-\infty}). \]
However, this is the same as the full Bergman kernel, up to another $O(k^{-\infty})$ error term. 

{\em (Forbidden Region).} If $z$ is in the forbidden region, $H(z) > E$, then only terms with $|H(z)-j/k|$ small will contribute. Recall as in Definition \ref{d:zebe}, we define $j_k = \max \{ \Z \cap k[0,E)\}$ and $E_k = j_k/k$. Let $z_j \in H^{-1}(j/k)$ and $\tau_j$ be such that, $z = e^{\tau_j} z_j$.  Since $H(z) > E >  H(z_j)$, we have $\tau_j > 0$.  Then using Proposition \ref{pp:eqv-decay}, we have
\[ \frac{\Pi_{k,P}(z)}{e^{-k b(z, E_k)} \Pi_{k,j_k}(z_{j_k})} = \sum_{j \in kP \cap \Z} \frac{ \Pi_{k,j}(z)}{e^{-k b(z, E_k)} \Pi_{k,j_k}(z_{j_k})}  = \sum_{j \in kP \cap \Z} \frac{e^{-k b(z, j/k)} \Pi_{k,j}(z_j) }{e^{-k b(z, E_k)} \Pi_{k,j_k}(z_{j_k})} .\]
For any $1 \gg \epsilon >0$, we claim the following localization result
\be \label{forbidden-local}
 \frac{\Pi_{k,P}(z)}{e^{-k b(z,E_k)} \Pi_{k,j_k}(z_{j_k})} = \sum_{j \in kP \cap \Z} \frac{e^{-k b(z, j/k)} \Pi_{k,j}(z_j) }{e^{-k b(z, E_k)} \Pi_{k,j_k}(z_{j_k})} \chi\left( \frac{|j/k-E_k|}{k^{-1+\epsilon}}\right) + O(k^{-\infty}). \ee
{\em Proof of the claim:} By Taylor expansion of $b(z,E)$ in $E$, there exists $\delta, C>0$, such that $\forall |E-E_k| < \delta$
\be b(z,E) = b(z,E_k) + (E-E_k) \pa_E b(z,E_k) + R^{(2)}_b(z,E,E_k),  \quad
|R^{(2)}_b(z,E,E_k)| \leq C |E-E_k|^2 .\ee
Then if $(j - j_k)> k^{\epsilon}$, and $k$ large enough such that $k^{-1+\epsilon} < \delta$, then 
\[ k[b(z, j/k) - b(z, E_k)] >k[b(z, j + k^{\epsilon}/k) - b(z, E_k)] = \pa_E b(z, E_k) k^{\epsilon} + k R^{(2)}_b(z, E_k + k^{\epsilon-1},E_k)   .\]
Since $\pa_E b(z, E_k) = - 2 \tau(z, E_k) > 0$, and $k R^{(2)}_b(z, E_k + k^{\epsilon-1},E_k) < C k^{-1+2\epsilon} \ll k^\epsilon$, we have 
\[ \sum_{j \in kP \cap \Z}  \frac{e^{-k b(z, j/k)} \Pi_{k,j}(z_j) }{e^{-k b(z, E_k)} \Pi_{k,j_k}(z_{j_k})}  \left(1 - \chi\left( \frac{|j-j_k|}{k^{\epsilon}}\right) \right) = O(k^{-\infty}).\]
This finishes the proof of the localization claim \eqref{forbidden-local}. 

Next, we claim that the sum in  \eqref{forbidden-local} can be approximated by an infinite geometric series with $O(1/k)$ error. Indeed, using Proposition \ref{pp:eqv-onshell}, we have
\[ \frac{\Pi_{k,j}(z_j) }{ \Pi_{k,j_k}(z_{j_k})} = \sqrt{\frac{\pa_\rho^2\varphi(z_{j_k})}{\pa_\rho^2\varphi(z_{j})}} + O(1/k) = 1 + R_1(z_j , z_{j_k}) + O(k^{-1}) ,\]
where $|R_1(z_j,  z_{j_k})| < C |H(z_j) - E_k| = k^{-1} \cdot C |j - j_k| $. And from \eqref{e:pebe}, we have
\[ \frac{e^{-k b(z, j/k)} }{e^{-k b(z, E_k)}} = e^{- A (j_k - j) } (1+R_2(j,j_k)), \quad A = -\pa_E b(z, E_k) = 2 \tau(z,E_k), \]
where $R_2(j,j_k)  = e^{k R^{(2)}_b(z, j/k, E_k)} -1 < C k \cdot |j/k - E_k|^2 = k^{-1} \cdot C |j - j_k|^2$. Hence we get
\bea
&&  \sum_{j \in kP \cap \Z} \frac{e^{-k b(z, j/k)} \Pi_{k,j}(z_j) }{e^{-k b(z, E_k)} \Pi_{k,j_k}(z_{j_k})} \chi\left( \frac{|j/k-E_k|}{k^{-1+\epsilon}}\right) \\
&=&  \sum_{j \in kP \cap \Z} e^{- A (j_k - j)} (1 + k^{-1} R(j-j_k, z_0)) \chi\left( \frac{|j/k-E_k|}{k^{-1+\epsilon}}\right) 
 \\ 
 &=&  \left( \sum_{j \in kP \cap \Z}  e^{- A (j_k-j)}  \chi\left( \frac{|j/k-E_k|}{k^{-1+\epsilon}}\right)\right) (1+O(k^{-1})) \\
 &=&  \sum_{j \in kP \cap \Z}  e^{- A (j_k-j)} (1+O(k^{-1}))
 = (1 - e^{-A})^{-1}(1+O(k^{-1})),
\eea
where $R(m,z_0)$ has at most polynomial growth in $m$, hence is integrable against the exponential decaying factor.  Thus, we have proved 
\[
\Pi_{k,P}(z) = \Pi_{k,j_k}(z_{j_k}) \frac{e^{-k b(z, E_k)} }{1-e^{-2 \tau(z, E_k)}} (1+O(k^{-1})).
\]
Using \eqref{e:pikzk}, we have
\[
\Pi_{k,P}(z) = k^{m-1/2} \sqrt{\frac{2}{\pi \pa^2_\rho \varphi(z_k)}} \frac{e^{-k b(z, E_k)} }{1-e^{-2 \tau(z, E_k)}} (1+O(k^{-1})).
\]
Finally, one may replace $\pa^2_\rho \varphi(z_k)$ by $\pa^2_\rho \varphi(z_E)$ and $\tau(z, E_k)$ by $\tau(z,E)$ with an additional $(1+O(1/k))$ factor. 
This concludes the proof for Theorem \ref{thm:bulk}. 
\epf

\bpf [Proof of Theorem \ref{thm:interface}]
We write $z$ for the sequence $z_k = e^{\beta/ \sqrt{k}} \cdot z_E$, for point $z_E \in H^{-1}(E) \cap M^E_{\max}$. 
By the Localization Lemma \ref{lm:localization}
\bea
\Pi_{k,P}(z) &=& \sum_{j \in k P \cap \Z}  \Pi_{k,j} (z) \chi\left( \frac{|H(z) - j/k|}{k^{-1/2+\epsilon}} \right) + O(k^{-\infty}) \\
&=&\sum_{j \in k P \cap \Z} e^{-k b(z,j/k)}  \Pi_{k,j} (z_j) \chi\left( \frac{| H(z) - j/k|}{k^{-1/2+\epsilon}} \right)+ O(k^{-\infty})
\eea
Next we Taylor expand $b(z,E)$ around $E=H(z)$,  $ \exists \delta, C>0$, such that for all $|H(z) - E| < \delta$, we have
\[ b(z,E) = \frac{|E-H(z)|^2}{2} \pa_E^2 b(z,E) + R^{(3)}_b(z,E) = \frac{|E-H(z)|^2}{2} \frac{4}{\pa_\rho^2 \varphi(z)} + R^{(3)}_b(z,E) \]
where we have used Lemma \ref{lm:sympbE}, and $|R^{(3)}_b(z,E)| < C |E-H(z)|^3$. Define
\[ u_j = \sqrt{k}(j/k-H(z)) \]
we have 
\[ k b(z, j/k) = A_2 u_j^2+ k R^{(3)}_b(z,j/k),  \quad A_2 =\frac{2}{\pa_\rho^2 \varphi(z)} \]
and 
\[ k R^{(3)}_b(z,j/k) < C k^{-1/2} |u_j|^3 < C k^{-1/2+3\epsilon}.\] 
where  we used $\chi(|u_j|/k^\epsilon) >0$ only $|u_j| < 2 k^\epsilon$. We have
\bea
\frac{\Pi_{k,P}(z)}{\Pi_{k,j_0}(z_{j_0})} &=& \sum_{j \in k P \cap \Z} e^{-A_2 u_j^2}  e^{k R^{(3)}_b(z,j/k)} \cdot   \frac{\Pi_{k,j}(z_j) }{ \Pi_{k,j_0}(z_{j_0})} \cdot \chi(u_j/k^\epsilon) \\
&=& \sum_{j \in k P \cap \Z} e^{-A_2 u_j^2}(1+k^{-1/2} R(u_j)) \chi(u_j/k^\epsilon) \\
&=& \left(\sum_{j \in k P \cap \Z} e^{-A_2 u_j^2} \chi(u_j/k^\epsilon) \right)(1+O(k^{-1/2})) \\
&=& \left(\sum_{j \in k P \cap \Z} e^{-A_2 u_j^2}\right)(1+O(k^{-1/2})) 
\eea
where $R(u_j)$ has at most polynomial growth in $u_j$, hence is integrable against the Gaussian decaying factor, and removing the cut-off function in the last step only will introduce an   error of size $O(k^{-\infty})$. Finally, we replace the sum with the integral over $u$. Since $u_{j+1} - u_j = 1/\sqrt{k}$, and the integrand is smooth and has bounded derivative, the difference between the integral and the summation is again $O(k^{-1/2})$
\[
\frac{\Pi_{k,P}(z)}{\Pi_{k,j_0}(z_{j_0})} = \int_{\sqrt{k}(-H(z))}^{\sqrt{k}(E-H(z))}  \exp \left(-\frac{2 u^2}{ \pa_\rho^2 \varphi(z)}\right) \sqrt{k} du (1+O(k^{-1/2})) 
\]
Using our assumption that  $\sqrt{k}|E - H(z)| < C$,  we may extend the lower limit of the integral to $-\infty$, with an $O(k^{-\infty})$ error. Using Theorem \ref{E},  we can estimate the denominator as
\[ \Pi_{k,j_0}(z_{j_0}) = \Pi_{k,j_0}(z_E)  (1+O(k^{-1/2})) =  k^{m-\half} \sqrt{\frac{2}{\pi \pa^2_\rho \varphi(z_E)}} (1+O(k^{-1/2})). \] 
Then evaluate the incomplete Gaussian integral, we get
\bea
\Pi_{k,P}(z)
&=&k^m \Erf\left(\sqrt{\frac{4k}{\pa^2_\rho \varphi(z_E)}}(E - H(z))\right) (1+O(k^{-1/2})).
\eea
 Using  \eqref{prho2}, $\pa_\rho^2 \varphi = |\nabla H|^2 / \pi$, we finish the proof.  
\epf

\section{\label{LATTICE} Proof of Theorem \ref{thm:bulk}:  Euler-MacLaurin  approach}

In this section we prove Theorem \ref{thm:bulk} by the technique
of \cite{ShZ} of using  `polytope characters' to sift
out the weights in the given interval. In dimension one we refer to these
polytope characters as interval characters. They are very simple in dimension one and can be directly integrated. 

We recall our normalization has $H(M) = [0, E_{\max}]$.  Given a proper  subinterval $P = [0, E) \subset H(M)$ the 
{\it interval characters} $\chi_{k P}$  defined on 
$(\C^*)$ by
\begin{equation} \chi_{k  P} (e^w) =  \sum_{ j = 0}^{k E_k} e^{jw }  = \sum_{j \in \Z \cap [0, k E) } e^{jw}
\,\;\;\;w \in \C. \label{char}\end{equation}
where $E_k = \max \{ \frac{1}{k} \Z \cap [0, E) \}$ as in \eqref{Ek}.
The next Lemma expresses the partial Bergman kernel in terms of the interval character: 
\begin{lem} \label{CSCH} For any $z_1, z_2 \in M_{\max}$ close enough such that there is a local $\C^*$-invariant frame $e_L \in \Gamma(U,L)$,  with $U$ closed under $\C^*$-action and $U \ni  z_1, z_2$.  Define $K_{k,P}$ by $\Pi_{k,P}(z_1,z_2) = K_{k,P}(z_1,z_2) e_L^k(z_1) \ot \wb { e_L^{k}(z_2)}$.  For any proper  subinterval $P = [0, E) \subset H(M)$, and $a, b \in \R$ such that $a+b = 1$, we have the following contour integral expression
\bea 
K_{k,P} (z_1,z_2) &=&   \int_{|e^w|=1}
K_{k } ( e^{-a w} \cdot z_1, e^{-b \bar w} \cdot z_2) {\chi_{k P}(e^w)}\, \frac{ d w}{2\pi i};. 
\eea
\end{lem}
\bpf
From the equivariant property of $K_{k,j}(z,w)$ in Lemma \ref{lm:scalinglaw}, we have
\[ K_{k}( e^{-a w} \cdot z_1, e^{-b \bar w} \cdot z_2) = \sum_{j} K_{k,j} ( e^{-a w} \cdot z_1, e^{-b \bar w} \cdot z_2) = \sum_{j} e^{-j w} K_{k,j} (  z_1, z_2).  \] 
The contour integral then extracts the correct weights $j$ from $K_{k}$. 
\epf

The particular case we use is $a = b = \half$, where we have the following identity,
\begin{equation} \label{IDPI} K_{k,  P}(e^{\zeta} z, z) = K_{k,  P}(e^{\zeta/2} z, e^{\bar{\zeta}/2} z),
\;\; (\zeta \in \C^*). \end{equation}

The main result is that interval characters are given by oscillatory integrals
over $P$.

\begin{prop}\label{newchar} Let $P =  [0, E)  \subset H(M)$ be a proper  subinterval of $H(M)$.   Then,  the interval characters \eqref{char} are oscillatory integrals
$$\chi_{k P}(e^w) = L(w)k \int_{[0,E_k]} e^{k   w x }d x + \half (1 + e^{k E_k w})\;,\quad \mbox{for
all }\ w\in \C \RM \{ \pm 2 \pi i, \pm 4 \pi i,\cdots \},$$ where
$$L(w) = \frac{w/2}{\tanh w/2}. 
$$
\end{prop}

\begin{proof} 
We recall the Euler-MacLaurin formula for lattice interval sum, for any $[a,b] \subset \R$, $a,b \in \Z$, we have
\[ \sum_{n \in [a,b]} e^{n w}  =  L(w) \int_a^b e^{xw} dx +  \frac{e^{aw} + e^{bw}}{2}. \]
Then plug in $a=0, b=k E_k$ gives the desired result. The claim holds for all $w \in \R$ and by analytic continuation we get the desired results. 
%
\end{proof}
\brem
See \cite{KSW} and \cite{ShZ} for the generalization to character sum over simple polytope. 
\erem
\ss{\label{FIRSTBSjAPP} Proof of Theorem \ref{thm:bulk}}
\bpf[Proof of Theorem \ref{thm:bulk}] In this section we use the Euler-MacLaurin formula and the Boutet de Monvel-\Sj parametrix discussed in section \ref{BSjSECT}. Although we allow $E$ to be any real
number, Proposition \ref{newchar} replaces $[0,E]$ by $[0, E_k]$ and we get integrals over the latter interval.

Combining Lemma \ref{CSCH} and Proposition \ref{newchar}, we obtain  the following representation. For any $\tau \in \R$, we have
\bee
\Pi_{k,P}(z) &=& e^{-k \varphi(z)} \int_{\tau-\pi i }^{\tau + \pi i } K_k(e^{-w/2}  z, e^{-\bar w/2}  z) \chi_{kP}(e^{w}) \frac{d w}{2 \pi i} \notag \\
 &=& \underbrace{e^{-k \varphi(z)} k \int_{\tau-\pi i }^{\tau + \pi i } \int_{[0,E_k]} K_k(e^{-w/2}  z, e^{-\bar w/2}  z) L(w) e^{k x w} \frac{dx d w}{2 \pi i}}_{I_1} + \underbrace{\half(K_{k,0}(z) + K_{k,kE_k}(z))}_{I_2}   \label{REP}
\eee
where in the first step, we used Lemma \ref{CSCH} with $a=1/2,b=1/2$
(or the identity \eqref{IDPI}), and shifted the integration contour from the unit circle to $|e^w| = e^{-\tau}$; this is possible because  the integrand is holomorphic in $w$ even for $C^{\infty}$ metrics.  We will see it is only necessary to shift the contour when $z$ lies in the forbidden region. We use $I_1, I_2$ to denote the integral term and the boundary term.

Using the parametrix  \eqref{BBSjeq}  for $K_k$, we  obtain
\[ I_1 =  e^{-k \varphi(z)} k^{m+1} \int_{\tau-\pi i }^{\tau + \pi i } \int_{[0, E_k]} e^{k \varphi(e^{-w/2} z, e^{-\bar w/2} z) + k x w}  L(w)A_k   \frac{dx d w}{2 \pi i} \]
where $A_k =  (1+O(k^{-1}))$ is a semi-classical symbol. 
The phase function is
\[ \Psi(w, x) := \varphi(e^{-w/2}  z, e^{-\bar w/2}  z) +  x w. \]

The asymptotics can be obtained in two ways. One is to apply stationary phase for oscillatory integrals with complex phase functions on the surface-with-boundary $S^1 \times [0, E_k]$. The details are quite different in the allowed and forbidden regions, but the overall argument is to locate critical points  $(w, x)$ satisfying
\begin{equation} \label{CPE} 0 = \pa_x \Psi =  w, \quad 0 = \pa_w|_{e^{-w} z = e^{-\bar w} z}  \Psi  = x - \half \pa_\rho \varphi(e^{-w/2} z)=
x - H(e^{-w/2} z)\end{equation}
and having maximal real part on the contour of integration.
 By \eqref{CD} they occur  near the diagonal $e^{-w} z = e^{-\bar w} z$, hence $\Im w=0$.
That is, at an interior critical point,
\[ w = 0, \quad x = H(z). \]
The second way is to remove the  $dx$ integral using,
\begin{equation} \label{INT0E} \int_{[0, E]}  e^{ k x w} dx = \frac{e^{ k E w} - 1}{ k E w}.  \end{equation}
This formula is not useful when $w = 0$, which is a critical point for the integral in the allowed region. But it is useful when $w \not= 0$, which 
is true of critical points for the integral when $z$ lies
 in the forbidden region.
We now give the details.

{\bf Allowed Region.}
We assume that  $z$ is in the allowed region and that it is not a critical
point of $H$.  We set $\tau = 0$ in \eqref{REP}. The asymptotics of $I_1$ 
are  thus determined by interior critical points $(w, x)$.

  The  critical point equations \eqref{CPE} force  $x=H(z) $ and $w = 0$. The Hessian matrix for $\Psi(\phi,x)$ ($w = i \phi$) at $(\phi,x)=(0, H(z))$ is 
\[ \pa_{xx} \Psi = 0, \quad \pa_{x w} \Psi = 1, \quad \pa_{ww}\Psi = - \frac{1}{4} \pa_{\rho}^2 \varphi(z). \]
The Hessian determinant in $(x, \phi)$ equals $1$ .
 Note that $\Psi(0, H(z)) = \varphi(z)$ at the critical point, so the phase factor cancels the pre-factor $e^{- k \phi}$ in $I_1$.  Hence the interior  stationary phase formula  (\cite{Ho}) gives
\[ I_1 = k^m L(0) (1+O(k^{-1})) =  k^m (1+O(k^{-1})). \]
where the integration of $d w d x$ gives a factor of $k^{-1}$. 

To complete the proof we show that  $I_2 = O(k^{-\infty})$. Indeed, there exists constant $c>0$, such that $|H(z) - 0| > c, |H(z) - E_k| > c$. By Theorem \ref{E}, $I_2 = O(k^{-\infty})$.  Since the computation would be the same, if we had replaced $[0,E_k]$ by $[0, E_{\max}]=H(M)$, we get $K_{k,P}(z) = K_k(z)$ for $z$ in the allowed region.
%

{\bf Forbidden Region.} 
In the  forbidden region, we have $H(z) > E$ and $z \in M_{\max}^E$. The phase of the integral $I_1$ has no critical points on the unit circle $|e^w|=1$ and we must deform the integral to pick up the dominant critical point. The relevant value of $\tau$ is $\tau = 2 \tau_E(z)$, where as above  $z = e^{\tau_{E}(z)} \cdot z_{E}$ for $z_{E} \in H^{-1}(E)$, $\tau_{E}(z) > 0$. Then $w \not=0$ on the contour and we can use \eqref{INT0E} to remove the $dx$ integral. 
 Then the dominant critical point is on the boundary  where $x = E_k$.  The real part of the  phase is smaller at $0$ than $E$, and is therefore negligible.  The critical equations are  $\pa_w \Psi=x-H(e^{-\tau_{E_k}(z)} z) = x - E_k=0$ is satisfied on the right boundary of $[0,E_k]$.

$I_1$ then can be explicitly written as
\[ I_1 = e^{-k \varphi(z)} k^{m+1} \int_{-\pi}^{\pi} \int_{[0, E_k]} e^{k \varphi( e^{-i\phi/2}  z_E, e^{i \phi/2} z_E) + 2 k x \tau_{E_k}(z) + i k x \phi} L(2 \tau_{E_k}(z) + i \phi) (1+O(k^{-1})) \frac{dx d \phi}{2\pi}  \]
 Alternatively, by \eqref{INT0E}, the  $dx$ integral  equals $\frac{e^{2 k E \tau_E(z) +  i k E_k \phi} - 1}{2 k E_k \tau_E(z) +  i k E_k  \phi}$. Since $e^{2 k E_k \tau_E(z)}$ is exponentially larger than $1$, we may absorb the second term of the numerator into the remainder estimate.  The integral of $d\phi$ contributes $1/\sqrt{k \cdot \frac{1}{4} \pa_\rho^2 \varphi(z_E)}$. Recalling the definition of $b_E$ (Definition \ref{d:zebe}), we obtain
\bee
I_1 &=& k^{m-1/2} e^{-k (\varphi(z) - 2 E_k \tau_{E_k}(z) - \varphi(z_E))} \frac{1}{\sqrt{2\pi}} \frac{1}{2 \tau_{E_k}(z)}  \frac{2}{\sqrt{\pa_\rho^2 \varphi(z_E)}}\frac{\tau_{E_k}(z)}{\tanh(\tau_{E_k}(z)}(1+O(k^{-1})) \notag \\
&=& k^{m-1/2}  e^{-k b_{E_k}(z)}  \frac{\sqrt 2}{\sqrt{\pi \pa_\rho^2 \varphi(z_E)}} \left( \frac{1}{1-e^{-2\tau_{E_k}(z)}} - \frac{1}{2} \right) (1+O(k^{-1})). \label{e:I1}
\eee
On the other hand, we can estimate the boundary terms
\be \label{I2} I_2 = \half K_{k, kE_k}(z) + O(k^{-\infty}) = k^{m-1/2} e^{-k b_{E_k}(z)} \frac{1}{\sqrt{2\pi \pa_\rho^2 \varphi(z_E)}} (1+O(k^{-1}))  \ee
 Combining $I_1, I_2$, we see the $-1/2$ term in the parenthesis in \eqref{e:I1} cancels out, and we get the result in the forbidden region. 
\epf

\section{\label{INTERSECT} Interface asymptotics: Proof of Theorem \ref{thm:interface} }

In Proposition \ref{INTERFACEf}, we first prove  a smoothed version of
Theorem \ref{thm:interface} in which the characteristic function ${\bf 1}_{[-\infty, E]}$  is
replaced by a Schwartz  test function $f \in \scal(\R)$\footnote{ $\scal(\R)$ denotes Schwartz space}. A density argument using the localization results of Section \ref{LOCALSECT} then extends the asymptotic result  from $f \in \scal(\R)$ to the characteristic function of any interval in $\R$ (Theorem \ref{p:weak}), and in particular proves
the leading order asymptotics stated in  Theorem \ref{thm:interface}. In Section \ref{EMSECT} we use the Euler-MacLaurin formula to obtain the stated remainder estimate for the  intervals $[-\infty, E]$ (Proposition \ref{EMINTERFACE}). The three main results of this section (Proposition \ref{INTERFACEf}, Theorem \ref{p:weak} and Proposition 
\ref{EMINTERFACE})  are more or less independent: each  uses a different techniques and sheds a different light
on the Erf-asymptotics of Theorem \ref{thm:interface}.

Recall the setup of Theorem \ref{thm:interface}, let $E$ be a regular value of $H:M \to \R$, $z_E \in H^{-1}(E)$ and $z_k = e^{\beta / \sqrt{k}} \cdot z_E$ for some constant $\beta$. We define a sequence of measures,
\begin{equation} \label{dmukzEbeta} d \mu_{k, z_E, \beta}(x) = \frac{1}{\Pi_k(z_k,z_k)}\sum_{j} \Pi_{k, j}(z_k) \delta_{ \sqrt{k} (j/k - E)}(x),\;\; (k = 1, 2, 3, \dots)
\end{equation}
and the  purported limiting measure,
\begin{equation}  \label{dmuinftyzEbeta}  d \mu_{\infty, z_E, \beta}(x) = e^{-\frac{1}{2} \left(\frac{2x}{\sqrt{\pa_\rho^2\varphi(z_E)}} - \beta \sqrt{\pa_\rho^2\varphi(z_E)}\right)^2} \frac {2 dx}{\sqrt{2 \pi \pa_\rho^2\varphi(z )}}. \end{equation}
In the following, we fix $E, z_E, \beta$ and  write $\mu_{k}$ and $\mu_\infty$ for $\mu_{k, z_E, \beta}$ and $ \mu_{\infty, z_E, \beta}$, respectively.

For any bounded continuous function $f\in C_b(\R)$, we define
\begin{equation} \label{Ifk} I_{k,f}(z):= k^{-m} \sum_{j} f(\sqrt{k}(\frac{j}{k} - E)) \Pi_{k,j}(z).   \end{equation}
Since $\Pi_k(z_k,z_k) = k^m(1+O(k^{-1/2}))$, we have $$ \int f d \mu_k =  I_{k,f}(z_k) (1 + O(k^{-\half})). $$

\subsection{Schwartz test functions}

Although we state the main result for Schwartz test functions, it is easily seen that much less is required of the test functions for  the asymptotics to be valid.

\bp  \label{INTERFACEf} With the same notation as in Theorem \ref{thm:interface}, and $f \in \scal(\R)$, we have
\[ I_{k,f}( e^{\beta/\sqrt{k}} \cdot z_E) = \int_{-\infty}^\infty f(x) e^{-\frac{1}{2} \left(\frac{2x}{\sqrt{\pa_\rho^2\varphi(z_E)}} - \beta \sqrt{\pa_\rho^2\varphi(z_E)}\right)^2} \frac {2 dx}{\sqrt{2 \pi \pa_\rho^2\varphi(z_E)}} +O_f(k^{-1/2}).\]
where the constant in $O_f$ depends on $f$. 
\ep

\bpf 
By the Fourier inversion formula,  we have
\bea \label{Izk}
I_{k,f}(z) & = &    k^{-m}\int_{\R} \hat{f}( t) e^{- i E \sqrt{k} t} 
K_{k}(e^{i t/\sqrt{k}} z,z) e^{-k\varphi(z)} dt. 
\eea
 We note that $t \to
\Pi_{k}(e^{i t/\sqrt{k}} z,z) $ is $2 \pi \sqrt{k}$-periodic (similarly
for the parametrix and remainder terms),  so the integrals converge
when $\hat{f} \in \scal(\R)$. We periodize $g(t) =\hat{f} e^{- i E t \sqrt{k} } $ by means
of the $\sqrt{k}$-periodization operator
\[ \pcal_{\sqrt{k}} g(t) := \sum_{\ell \in \Z} g(t+2\pi \sqrt{k} \ell), \;\; g \in \scal(\R), \]
which is periodic of period $2 \pi \sqrt{k}$. In fact the right side  converges
as long as $|g(t) | \leq C (1 + |t|)^{-1 - \epsilon}$. We write
$$ \pcal_{\sqrt{k}}(\hat{f} e^{- i E t \sqrt{k} } ) 
= \sum_{\ell \in \Z} \hat f(t+2\pi \sqrt{k} \ell) e^{- i E t \sqrt{k} - 2 \pi i k \ell E} =: e^{-i E t \sqrt{k}} \hat{F}_k(t), $$
 with $\hat F_k(t) =  \sum_{\ell \in \Z} \hat f(t+2\pi \sqrt{k} \ell) e^{-2\pi i (k \ell  E)}.$
Then,
\bea \label{Izk2}
I_{k,f}(z) & = & k^{-m} \int_{- \pi \sqrt{k}}^{\pi \sqrt{k}} \hat F_k(t)  e^{-i E t \sqrt{k} }
K_{k}(e^{i t/\sqrt{k}} z,z) e^{-k\varphi(z)} dt.
 \eea

We then localize the last  integral using a smooth cutoff $\chi(\frac{t }{(\log k)^2} )$, where $\chi \in C_0^{\infty}(\R)$  is supported in $(-1,1)$ and equals to $1$ in $(-1/2, 1/2)$.  When $\pi \sqrt{k} \geq |t| \geq  (\log k)^2$, the
off-diagonal Bergman kernel $K_{k}(e^{i t/\sqrt{k}} z,z)$ is rapidly decaying   at the rate
$O(e^{- (\log k)^2})$. Here, we use the standard off-diagonal estimate,  $|K_k(x,y)|
\leq C k^{m} e^{- \beta \sqrt{k} d(x,y)}$ for certain $\beta, C > 0$ (see Theorem \ref{AGMON1} of the Appendix). Hence,
\bea \label{Izk3}
I_{k,f}(z) & = & k^{-m} \int_{\R} \chi(\frac{t }{(\log k)^2} )\; \hat F_k(t)  e^{-i E t \sqrt{k}}
K_{k}(e^{i t/\sqrt{k}} z,z) e^{-k\varphi(z)} dt + O_f(k^{-\infty}), \eea
where the constant in $O_f(k^{-\infty})$  depends on $\|\hat{F}_k\|_{L^1(-\sqrt{k}, \sqrt{k})} =  \|\hat{f}\|_{L^1}$.

We then introduce the Boutet-de-Monvel-\Sjostrand parametrix for $K_k$, 
 \bea \label{Izk}
I_{k,f}(z) 
&=&   \int_{-\infty}^\infty  \chi(\frac{t }{(\log k)^2} )\; \hat F_k(t)    e^{-i E t \sqrt{k}}  e^{k \varphi(e^{i t/\sqrt{k}} \cdot z, z) - k \varphi(z)} A_k(e^{i t/ \sqrt{k}} z,z)  dt  
 \\
 &+&  \int_{-\infty}^\infty \chi(\frac{t }{(\log k)^2} )\; \hat F_k(t)    e^{-i E t \sqrt{k} } R_k(e^{i t/ \sqrt{k}} z,z)  dt + O_f(k^{-\infty}).
\eea
By the parametrix construction,   $R_k \in k^{-\infty} C^{\infty}(M \times M)$, hence  the second term is
$O(k^{-\infty})$ and may be absorbed into the remainder estimate.

As in \eqref{PSIDEF}, the phase function of  $I_{k,f}$ is 
\begin{equation} \label{PSIDEF2} \Psi(it,  z) = -it (\sqrt{k}E) + k \varphi(e^{it / 2 \sqrt{k}} \cdot z,e^{-{it} /2 \sqrt{k}} \cdot z) - k \varphi(z). \end{equation} 
Recall that  $z_k = e^{\beta/\sqrt{k}} z_E$ with $H(z_E) = E$.  
 Then
as $k \to \infty$,
\bea
\Psi(it,  e^{\beta/\sqrt{k}} z_E) & = &  -it (\sqrt{k}E) + k\left( \varphi(e^{(it / 2 +\beta)/\sqrt{k}} \cdot z_E,e^{(-it/ 2 + \beta)/ \sqrt{k}} \cdot z_E) -  \varphi(e^{\beta/\sqrt{k}} \cdot z_E) \right)\\
&=&  -it (\half\sqrt{k}  \pa_\rho \varphi(z_E) )  + k \left[ \left(\frac{it / 2 +\beta}{\sqrt{k}}\right) \pa_\rho \varphi(z_E) + \half \left(\frac{it / 2 +\beta}{\sqrt{k}}\right)^2 \pa_\rho^2\varphi(z_E) \right. \\ 
&& \left. -  \left(\frac{\beta}{\sqrt{k}}\right) \pa_\rho \varphi(z_E) - \half \left(\frac{\beta}{\sqrt{k}}\right)^2 \pa_\rho^2\varphi(z_E) \right] + g_3(it, z, \beta) \\
&=&   \half( (it / 2 +\beta)^2 - \beta^2)\pa_\rho^2\varphi(z_E) +  g_4(it, z, \beta),  
\eea
where
\be g_3 = O(k^{-1/2} (|\beta|^3+|t|^3)),\;\; g_4 =  O(k^{-1/2} (|\beta|^3 + |t|^3)). \label{g3g4}\ee

We substitute the Taylor expansion into the phase of  the first term of $I_{k, f}(e^{\beta/\sqrt{k}} z_E)$,
and also Taylor expand $e^{ g_4}$ to order $1$.  Let $e_1(x) = 1 - e^x$. Since $|t| \leq (\log k)^2$  on the support of the integrand, $|g_4| \leq  C (\frac{(\log k)^6}{\sqrt{k}})$ on $|t| \leq (\log k)^2$.  Since $e^x = 1 + e_1(x)$ where $e_1(x) \leq 2 x$ on $[0,  C (\frac{\log k)^6}{\sqrt{k}})]$,  $e^{g_4} = 1 + \tilde{g}_4$ where $\tilde{g}_4(k,t) \leq 2 g_4 \leq C_0 k^{-\half} (1 + t^3)$ on $ [0, (\log k)^2]$.   

We get
\bea  I_{k,f}(e^{\beta/\sqrt{k}} z_E) &
= &   \int_{\R} \chi(\frac{t}{ (\log k)^2})\hat F_k(t) e^{ \half( (it / 2 +\beta)^2 - \beta^2)\pa_\rho^2\varphi(z_E)}  (1 + \tilde{g}_4)) dt + O_f(k^{-1/2}) \\&& \\
& = & \int_{\R} \chi(\frac{t}{(\log k)^2})\hat F_k(t) e^{ \half( (it / 2 +\beta)^2 - \beta^2)\pa_\rho^2\varphi(z_E)}  dt + O_f(k^{-1/2})
\eea
where $ \chi(\frac{t}{ (\log k)^2}) |\wt g_4| \leq C_0 k^{-1/2} (1+ |t|^3)$ after integration against the Gaussian factor is of size $O(k^{-1/2})$.

Finally, we unravel the periodization $\hat{F}_k$ to  evaluate the first term.
\bea
&&\int_{\R} \chi(\frac{t}{(\log k)^2})\hat F_k(t) e^{ \half( (it / 2 +\beta)^2 - \beta^2)\pa_\rho^2\varphi(z_E)}  dt \\
&=& \int_{\R} \chi(\frac{t}{(\log k)^2})\hat f(t) e^{ \half( (it / 2 +\beta)^2 - \beta^2)\pa_\rho^2\varphi(z_E)} dt   \\
& + &\sum_{\ell \in \Z \RM 0} \int_{\R} \chi(\frac{t}{(\log k)^2})\hat f(t + 2 \pi \sqrt{k} \ell) e^{2\pi i k \ell E + \half( (it / 2 +\beta)^2 - \beta^2)\pa_\rho^2\varphi(z_E)} dt \\
&=& \int_{\R} \chi(\frac{t}{(\log k)^2})\hat f(t) e^{ \half( (it / 2 +\beta)^2 - \beta^2)\pa_\rho^2\varphi(z_E)} dt  + O_f(k^{-\infty})
\eea
where in bounding the terms with $\ell \neq 0$, we have used the fast decay property of the Schwarz function $\hat f(t)$, i.e., for any positive integer $N$, we have $|\hat f(t + 2\pi \sqrt{k} \ell)| < C_N (1+|t + 2\pi \sqrt{k} \ell|)^{-N}$ for some $C_N$, hence the sum over $\ell$ is convergent by $l^{-N}$ factor. 

Finally, removing the cut-off $\chi(t/(\log k)^2)$ will introduce an error as $\int_{(\log k)^2}^\infty e^{-a x^2} dx = O(k^{-\infty})$. We have
\bea  I_{k,f}(e^{\beta/\sqrt{k}} z_E) &=& \int_{\R} \hat f(t) e^{ \half( (it / 2 +\beta)^2 - \beta^2)\pa_\rho^2\varphi(z_E)}  dt  + O_f(k^{-1/2}) \\
&=&  \int_{-\infty}^\infty f(x) e^{-\frac{1}{2} \left(\frac{2x}{\sqrt{\pa_\rho^2\varphi(z_E)}} - \beta \sqrt{\pa_\rho^2\varphi(z_E)}\right)^2} \frac {2 dx}{\sqrt{2 \pi \pa_\rho^2\varphi(z_E)}} +O_f(k^{-1/2})
\eea
by the Plancherel theorem.  This completes the proof of Proposition \ref{INTERFACEf}.
\epf

\subsection{Proof of Weak Convergence result.}

We now use Proposition \ref{INTERFACEf} to prove the following weak-convergence result:

\begin{theo} \label{p:weak}
The sequence of measures $\mu_{k }$ converges to $\mu_{\infty}$ weak*
on $C_b(\R)$. In particular, for any interval $I$,  possibly unbounded, 
\[ \mu_{k}(I) \to \mu_{\infty}(I). \]
\end{theo}

This proves the leading order convergence statement of Theorem \ref{thm:interface}.  We first prove that 
$ \int_\R f(x) d \mu_{k}(x)  \to \int_\R f(x) d \mu_{\infty}(x)$ for $f \in C_b(\R)$.
We then use the results of Section \ref{LOCALSECT} to show that   $\{\mu_k\}_k$ is a tight  family of probability measures. 

\begin{proof}

\begin{lem} \label{p:Cc}
For any $f \in C_c(\R)$\footnote{$C_c(\R)$ denotes continuous functions of compact support.}, we have
\[ \lim_{k \to \infty} \int f d \mu_k = \int f d \mu_\infty. \]
\end{lem}
\bpf
Let $\eta(x)$ be a smooth non-negative compactly supported function, such that $\int \eta(x) dx = 1$. Let $\eta_\epsilon(x) = \epsilon^{-1} \eta(x/\epsilon)$, and $f_\epsilon = \eta_\epsilon \star f$. Then $f_\epsilon \to f$ in the $C^0$-norm and $f_\epsilon \in C^\infty_c(\R)$. Given  $\delta>0$, choose $\epsilon$  small enough such that $|f_\epsilon - f|_{C^0}< \delta$. Then, 
\[ \left|  \int f (d \mu_k - d \mu_\infty )\right| \leq \left|   \int f_\epsilon (d \mu_k -  d \mu_\infty )\right|  +  \int |f-f_\epsilon| d \mu_k  + \int|f-f_\epsilon| d \mu_\infty \leq  \left|   \int f_\epsilon (d \mu_k -  d \mu_\infty )\right|  + 2\delta \]
By Proposition \ref{INTERFACEf},  $ \left|   \int f_\epsilon (d \mu_k -  d \mu_\infty )\right| < \delta$ for $k$ sufficiently large. 
Since $\delta$ is arbitrarily, Lemma \ref{p:Cc} follows.
\epf

Next,we prove that the sequence of measure $\mu_k$ is tight and extend the range of the test function from $C_c(\R)$ to $C_b(\R)$. 
\begin{lem}
The sequence of measure $\{\mu_k\}$ is tight and,  for any $f \in C_b(\R)$, 
\[ \lim_{k \to \infty} \int f d \mu_k = \int f d \mu_\infty. \]
\end{lem}
\bpf
To prove tightness, for any $\epsilon>0$, we need to find $R>0$ large enough, such that   $\mu_k( \R \RM [-R, R]) < \epsilon$ for all $k$. The existence of such $R$ is
an immediate consequence of  Lemma \ref{lm:localization} (3) on localization of sums.

We then prove weak convergence:
Let $\epsilon, R$ be as above. Let $\chi(x)$ be a cut-off function that equal to $1$ on $[-R, R]$ and equals to zero for $|x|>R+1$. Then
\[ \left|\int f (d\mu_k - d \mu_\infty) \right| \leq \left|\int f \chi(x) (d\mu_k - d\mu_\infty) \right| + \left|\int f (1-\chi) d \mu_k \right| + \left| \int f (1-\chi) d \mu_\infty\right|. \]
The last two terms can be bounded by $2 \epsilon \|f\|_{C^0}$, and the first term tends to $0$ as $k \to \infty$ since $f \chi \in C_c(\R)$. Thus 
\[ \lim_{k \to \infty} \left|\int f (d\mu_k - d \mu_\infty) \right| \leq 2 \epsilon \|f\|_{C^0}\]
for all $\epsilon$, and the left hand side has to be zero. This finishes the proof of the Lemma and hence  the proof of Proposition \ref{p:weak}. 
\epf

\end{proof}

\subsection{\label{EMSECT} Proof by the Euler-MacLaurin method} In this section we use
the Euler-MacLaurin method of Section \ref{FIRSTBSjAPP} to obtain a remainder estimate for the weak convergence, as claimed in 
Theorem \ref{thm:interface}.

Define
 \begin{equation}\label{PkEdef}
  I_{[-\frac{M}{\sqrt{k}}, \frac{M}{\sqrt{k}}]}(e^{\frac{\beta}{\sqrt{k}}}
  z_E) := k^{-m} \sum_{j: |\frac{j}{k} - E|  \leq \frac{M}{\sqrt{k}}} \Pi_{k,j} (e^{\frac{\beta}{\sqrt{k}}}
  z_E) \simeq \mu_{k, z_E, \beta} [-M, M].   \end{equation}
  These are sums of the type \eqref{Ifk} but with $f = {\bf 1}_{[-M, M]}$.
  As above, we use  that  $\Pi_k(z_k,z_k) = k^m(1+O(k^{-1/2}))$ to
  normalize by the simpler factor $k^{-m}$. The following Proposition (with trivial modification from $[-M,M]$ to $(-\infty,M]$)
  implies Theorem \ref{thm:interface} (with the remainder estimate). For the sake of brevity, we omit further details.
  
  \begin{prop} \label{EMINTERFACE} Let $z_E \in H^{-1}(E)$ and fix
  real numbers $M>0, \beta \in \R$. Then 
  $$ I_{[-\frac{M}{\sqrt{k}}, \frac{M}{\sqrt{k}}]}(e^{\frac{\beta}{\sqrt{k}}}
   z_E) =     \int_{-M}^M \sqrt{\frac{2}{\pi \pa_\rho^2\varphi(z_E)}}e^{-\frac{(2y -  \beta \pa_\rho^2\varphi(z_E))^2}{2\pa_\rho^2\varphi(z_E)}}(1+ O(k^{-\half}))dy, $$
  
  \end{prop}
  
  \begin{proof}
  We use  Proposition \ref{newchar} with $P = [E - \frac{M}{\sqrt{k}}, E + \frac{M}{\sqrt{k}}]$ and \eqref{REP} with $z_k = e^{\frac{\beta}{\sqrt{k}}} z_E$ to get 
  $ I_{[-\frac{M}{\sqrt{k}}, \frac{M}{\sqrt{k}}]}(z_k):= I_1 + I_2$ with  
\[ I_1 =  e^{-k \varphi(e^{\frac{\beta}{\sqrt{k}}}\, z_E)} k \int_{-i\pi  }^{i \pi  } \int_{E-\frac{M}{\sqrt{k}}}^{E +  \frac{M}{\sqrt{k}}} e^{k \varphi(e^{-w/2} e^{\frac{\beta}{\sqrt{k}}}\, z_E, e^{-\bar w/2} e^{\frac{\beta}{\sqrt{k}}}\, z_E) + k x w}  L(w)A_k   \frac{dx d w}{2 \pi i} \]
where as above $A_k =  (1+O(k^{-1}))$ is a semi-classical symbol and where
we omit the boundary term \eqref{I2} $I_2 =  \half(K_{k, k[E -  \frac{M}{\sqrt{k}}]}(e^{\frac{\beta}{\sqrt{k}}}\, z_E)+ K_{k,k[E +  \frac{M}{\sqrt{k}}]}(e^{\frac{\beta}{\sqrt{k}}} z_E ))$ since it has  lower order, indeed the first sum $I_1$ having $O(\sqrt{k})$ terms of almost constant order  and the boundary term $I_2$ having only two of the same order.

We change variables in the $dx$ integral to $x =  E + \frac{y}{\sqrt{k}}$ so that
the $e^{k x w} dx$ integral becomes $\frac{e^{k w E}}{ \sqrt{k}}  e^{\sqrt{k} y w} dy. $
The full ($k$-dependent)  phase function becomes
\[ \Psi(w, y) := k \varphi(e^{-w/2}  e^{\frac{\beta}{\sqrt{k}}}\, z_E, e^{-\bar w/2}  e^{\frac{\beta}{\sqrt{k}}}\, z_E)  -  k \varphi(e^{\frac{\beta}{\sqrt{k}}}\, z_E)+ k  E w + \sqrt{k} y w. \]
We then
change variables to $t = i  \sqrt{k} w$ to obtain a new phase resembling
\eqref{PSIDEF2},
\bea  \label{PSIDEF3}
\Psi(it,  y ) & = & i y t  -it (\sqrt{k}E) + k\left( \varphi(e^{(it / 2 +\beta)/\sqrt{k}} \cdot z_E,e^{(\bar{it} / 2 + \beta)/ \sqrt{k}} \cdot z_E) -  \varphi(e^{\beta/\sqrt{k}} \cdot z_E) \right)\\
&=&iy t  -it (\half\sqrt{k}  \pa_\rho \varphi(z_E) )  + k \left[ \left(\frac{it / 2 +\beta}{\sqrt{k}}\right) \pa_\rho \varphi(z_E) + \half \left(\frac{it / 2 +\beta}{\sqrt{k}}\right)^2 \pa_\rho^2\varphi(z_E) \right. \\ 
&& \left. -  \left(\frac{\beta}{\sqrt{k}}\right) \pa_\rho \varphi(z_E) - \half \left(\frac{\beta}{\sqrt{k}}\right)^2 \pa_\rho^2\varphi(z_E) \right] + g_3(it, z, \beta) \\
&=& iy t +  \half( (it / 2 +\beta)^2 - \beta^2)\pa_\rho^2\varphi(z_E) +  g_4(z, it, \beta),  
\eea
where
$$g_3 = O(k^{-1/2} (|\beta|^3+|t |^3),\;\; g_4 =  O(k^{-1/2} (|\beta|^3 + |t|^3)), $$
and the ranges are $t \in [-\pi \sqrt{k}, \pi \sqrt{k}]$ and $y \in [-M, M]$. 

We further cutoff the integrand to a $(\log k)^2$-neighborhood of
$t = 0$ using a smooth cutoff $\chi(\frac{ t }{(\log k)^2})$ where $\chi \equiv 1$ in a neighborhood of $t = 0$ and $\chi \equiv 0$ outside a slightly larger neighborhood, and observe that the part of the integral with the cutoff $(1 - \chi(\frac{t}{(\log k)^2}))$ is rapidly decaying in $k$ (c.f. the argument after \eqref{g3g4}).
Substituting  $\tau=it$ and integrating $dt$ gives,
\bea
I_1 &\simeq&   \int_{-M}^M \int_{-\infty}^{\infty} \chi(\frac{t}{(\log k)^2}) e^{ \half( (it / 2 +\beta)^2 - \beta^2)\pa_\rho^2\varphi(z_E)}  e^{i y t}   (1+O(k^{-\half})) \frac{dt  dy}{2\pi} \\&&\\
&\simeq &    \int_{-M}^M \int_{-\infty}^{\infty} e^{ \half( (it / 2 +\beta)^2 - \beta^2)\pa_\rho^2\varphi(z_E)}  e^{i y t}   (1+O(k^{-\half}))  \frac{dt  dy}{2\pi}  \\
&\simeq&  \int_{-M}^M \sqrt{\frac{2}{\pi \pa_\rho^2\varphi(z_E)}}e^{-\frac{(2y -  \beta \pa_\rho^2\varphi(z_E))^2}{2\pa_\rho^2\varphi(z_E)}}(1+ O(k^{-\half}))dy,
\eea
where $\simeq$ denotes asymptotics as $k \to \infty$.

\end{proof}

\section{Distribution of zero locus of a Random section: Proof of Theorem \ref{Zth}}
First we recall a proposition that links the expectation of the $(1,1)$ current defined by a random section from the Hilbert subspace $\scal_k \subset H^0(M, L^k)$. Recall our setup: Let $s =\sum_{j=1}^{\dim \scal_k} a_{k,j} s_{k,j}$ where $a_{k,j}$ are i.i.d. complex $N(0,1)$ random variables and $\{s_{k,j}\}$ is an orthonormal basis of $\scal_k$. Let $Z_s $ be the zero set of $s$ and let $[Z_s]$ be the current of integration over
$Z_s$.
\bp [\cite{ShZ}, Proposition 4.1] \label{p:zero}
\be \frac{1}{k} \E ([Z_s]) = \frac{\sqrt{-1}}{2\pi k} \ddbar \log \Pi_{\scal_k}(z) + c_1(L, h)\ee
where $\Pi_{\scal_k}(z)$ is the partial Bergman density function. 
\ep

The random zero locus in the interior of the allowed region is uniformly distributed, as if $\scal_k = H^0(M, L^k)$, indeed $\Pi_{\scal_k}(z) = 1 + O(k^{-1})$ is approximately constant. Our main interest is the distribution in the forbidden region. 

As before, we take $\acal = \{z \mid H(z) < E\}$ and $\fcal = \{z \mid H(z) > E\}$ for some regular value $E$ of $H$. Let $\fcal^E_{\max}$ be the open dense subset of $\fcal$ where the $S^1$-action acts freely, and where the $\R_+$-orbit of $z$ intersect $H^{-1}(E)$, say at $q_E(z)$. We define $\pi_E: H^{-1}(E) \cap M_{\max} \to X_E$ to be the Hamiltonian reduction of $H^{-1}(E)\cap M_{\max}$ by the $S^1$-action. Then we define another projection map
\[ \wb q_E: \fcal^E_{\max} \to X_E, \quad \wb q_E(z) = \pi_E \circ q_E(z). \]
The complex structure on the quotient is defined as the quotient of the semi-stable points by the $\C^*$ action, and   $\wb q_E$ is the restriction of this quotient map to $\fcal^E_{\max}$. Hence $\wb q_E$ is holomorphic by definition.

   Let $\varphi_E$ be the restriction of $\varphi$ to $H^{-1}(E)$. Since it is $S^1$-invariant,  it descends to a \Kahler potential on $X_E$ as well and is the \kahler potential of the reduced \kahler form $\omega_E$  on $X_E$. The following is a somewhat more precise version of Theorem \ref{Zth}:

\bp
For any compact subset $K \subset \fcal^E_{\max}$, we have the following weak* convergence
\be \lim_{k \to \infty} \frac{1}{k} \E ([Z_s]) = \wb q_E^{-1} (\omega_E) + 2 E \frac{\sqrt{-1}}{2\pi} (\ddbar \tau(z, E)) . \label{e:zero} \ee
In particular, the right hand side of \eqref{e:zero} is a smooth $(1,1)$-form of rank $(n-1)$ in $\fcal^E_{\max}$. \footnote{In the toric
case the leaves (orbits) of the $\C^*$ action vary holomorphically and $\ddbar \tau = 0$. }
\ep

\bpf
Using the expansion of the partial Bergman density in the forbidden region, we have for $z \in \fcal$, 
\[ \Pi_{kP}(z) =   k^{m-1/2} \sqrt{\frac{2}{\pi \pa_\rho^2 \varphi (z_E)}} \frac{e^{-kb(z,E_k)}}{1-e^{-|2\tau(z,E)|}}  (1+O(k^{-1})). \]
Using Proposition \ref{p:zero}, we get
\bea
 \frac{1}{k} \E ([Z_s]) &=& \frac{\sqrt{-1}}{2\pi} \ddbar (-b(z, E_k) + \varphi(z)) +O(k^{-1}) \\
 &=&  \frac{\sqrt{-1}}{2\pi} \ddbar (-b(z, E) + \varphi(z)) +O(k^{-1}) \\
 &=& \frac{\sqrt{-1}}{2\pi} \ddbar [\varphi(q_E(z)) + 2E \tau(z, E) ]+O(k^{-1})
 \eea
where we have used $c_1(L,h) = \frac{\sqrt{-1}}{2\pi} \ddbar \varphi(z)$ and expression of $b(z, E)$ from  \eqref{bEFORM}. 
Since $\varphi_E$ is the \Kahler potential on the symplectic reduction $H^{-1}(E)/S^1$, we have $\varphi(q_E(z))  = \varphi_E(\wb q_E(z))$,  and since $\wb q_E$ is holomorphic, we can commute $\ddbar$ with the pullback. This gives the desired result \eqref{e:zero}. For the rank statement, we note that the $(1,1)$-form on the right-hand-side of  \eqref{e:zero} vanishes on any $\C^*$-leaf inside $\fcal^E_{\max}$ hence is of rank $(n-1)$. 
\epf

\section{Example: The Bargmann-Fock model}

In this section we illustrate the results in the Bargmann-Fock model of the line bundle $\C^m \times \C \to \C^m$  with \kahler potential $\varphi = \|z\|^2$.
$$\hcal^2_k = \hcal^2_{h_{BF}^k} = \{f \in \ocal(\C^m): \int_{\C^m} |f(z)|^2 e^{- k \|z\|^2} d m(z) < \infty \} $$
where $dm = (\omega)^m/m!$ is Lebesgue measure, and $\omega = \frac{i}{2\pi} \ddbar \varphi = \frac{i}{2\pi} dz \wedge d\bar{z} = \pi^{-1} \sum_j dx_j \wedge dy_j$. As mentioned in \S \ref{EXSECT},  the  linear $S^1$ actions on $\C^m$ have the form,
$$e^{i \theta} \cdot (z_1, \dots, z_m) = (e^{i b_1 \theta} z_1, \dots, e^{i b_m \theta} z_m), 
\;\;\; b_j \in \Z, $$
with Hamiltonians 
$ H = \frac{1}{2}\pa_\rho|_{\rho=0} \sum_{j=1}^m ( |e^{b_j \rho} z_j|^2) =  \sum_j b_j |z_j|^2.$
We only consider the  diagonal $\T$ action  and 
Hamiltonian $H(z) =  \|z\|^2$, i.e. the isotropic harmonic oscillator in the Bargmann-Fock representation.

The usual quantum Hamiltonian for the harmonic oscillator is  $\hbar \left( \hat{N} + \frac{m}{2} \right)$ where $\hbar = 1/k$ where $\hat{N} = Z \cdot \frac{\partial}{\partial Z}$
is the number or Euler operator with eigenvalues/eigenfunctions
$$\hat{N} z^{\alpha} = |\alpha | z^{\alpha}. $$
where $|\alpha| := \sum_i \alpha_i$. Since we chose our normalization of $H$ to have minimum $0$, we will drop the $m/2$ constant, and define $H_k = \frac{1}{k} \hat{N}$.  It is an elliptic $S^1$ action in the sense that its moment map $H$ is proper and all weight spaces
$$\hcal_{k, j} = {\rm Span} \{z^{\alpha} = z_1^{\alpha_1} \cdots z_m^{\alpha_m}, |\alpha|=\alpha_1 + \cdots + \alpha_m = j\} $$  are finite dimensional.

We will fix the constant section $1 \in \Gamma(\C^m, \C)$ as the holomorphic reference frame, then an orthonormal basis is given by 
$$c_\alpha z^\alpha = \prod_{i=1}^m \sqrt{\frac{k^{\alpha_i+1}}{(\alpha_i)!}} z_i^{\alpha_i}$$ 
  thus the full Bergman kernel
$$\Pi_{k}(z,w) = K_{k} (z,w), \quad  K_{k}(z,w):  = \sum_{\alpha \in \Z_{\geq 0}^m}  \frac{k^{| \alpha |+m} z^{\alpha} \overline{w}^{\alpha}}{\alpha!} = k^{m} e^{k z\cdot \bar{w}} . $$
and the equivariant Bergman kernels are 
$$\Pi_{k,j}(z,w) = K_{k, j} (z,w), \quad  K_{k, j}(z,w): = \sum_{\alpha: | \alpha| = j}  \frac{k^{| \alpha |+ m} z^{\alpha} \overline{w}^{\alpha}}{\alpha!}. $$
The equivariant kernel is obtained from the full kernel by
\be \label{BF-avg} K_{k,j}(z,w) = \frac{1}{2\pi} \int_0^{2 \pi} e^{-i j \theta} K(e^{i \theta} z, w) d\theta \ee
And Bergman density $B_k(z)=k^m$, and the equivariant Bergman density is 
\[
\Pi_{k,j}(z) =   K_{k, j} (z,z) \|1^k(z)\|^2_{h^k(z)} =  e^{- k \|z\|^2} \sum_{\alpha: \| \alpha \| = j}  \frac{k^{m+| \alpha |} z^{\alpha} \overline{w}^{\alpha}}{\alpha!}\]

\begin{lem} 
As $k \to \infty$, and $E = j/k$,  the equivariant Bergman kernel is
$$K_{k, j}(z,z) = k^m \dashint_{\T} e^{- i j \theta} e^{k e^{i\theta} \|z\|^2} d\theta = k^{m}\frac{k^j}{j!} \|z \|^{2j} \simeq k^{m-1/2} \left( \frac{e \cdot \|z\|^2}{E} \right)^{kE} (2\pi E)^{-1/2} $$
and the equivariant Bergman kernel is
\[ B_{k,j}(z) =K_{k, j}(z,z) e^{-k \|z\|^2} = k^{m-1/2} (2\pi E)^{-1/2}  \left( \frac{ \|z\|^2}{E} \right)^{kE}  e^{-k (\|z\|^2-E)} \]
The maximum of $B_{k,j}(z)$ is obtained, when $\|z\|^2 = E$. 
\end{lem}

\begin{proof}

In fact, $K_{k, j}(z,w)$ is $U(m)$-invariant and so $K_{k,j}(z,w)$ is a function of $z \cdot \overline{w}$. It is also homogeneous of
degree $2j$ so it is a constant multiple $C=C_{k,j,m} $  of $(z \cdot \overline{w})^k$.  The constant may be determined from the fact that
$$ \dim V_k(j) = \int_{\C^m} B_{k,j}(z) d m(z) = \pi^{-m} \int_0^\infty e^{-kr^2}  \cdot C r^{2j}  \cdot r^{2m-1} \omega_{2m-1} dr $$ where $ \dim V_k(j) = {j+m-1 \choose m-1}$ is the number of partitions of $j$ in $m$ parts, $\omega_{d-1}=\frac{2 \pi^{d/2}}{\Gamma(d/2)} $ is the volume of $S^{d-1} \subset \R^d$. A straightforward computation gives $C_{k,j,m} = \frac{k^{m+j}}{j!}$. Hence $K_{k,j}(z,z)$ is the $j$-th term of the Taylor expansion $k^m e^{k \|z\|^2}$. 
But it is useful to compute the integral using the general method, which we will explain next. 

Let $E := E_{k,j}=\frac{j}{k}$. The first equality follows from \eqref{BF-avg}. We then change $\theta$ to $\theta + i \tau$ so that the complex phase is
$$\Psi_{z, \tau}(\theta) = - i E (\theta + i \tau) + e^{i (\theta + i \tau)} \|z\|^2. $$
The critical point equation is
 $$\frac{\partial}{i \partial \theta} \Psi_{z, \tau}(\theta) = - E +  e^{i (\theta + i \tau)} \|z\|^2 = 0 \iff E e^{- i \theta} = e^{ -\tau} \| z\|^{2}.  $$
 Since the right side is positive real, the only possible solution is  $\theta  = 0$ and for this we need to choose $\tau$ so that $e^{\tau} =\| z\|^{2}/E = H(z)/E$. With this choice of $\tau$,  and by deforming the contour to this $|w| = e^{\tau} \in \C$, the phase becomes $- i E(\theta + i \log (\| z\|^2/E) ) + Ee^{i \theta}$ and we have a non-degenerate critical point at $\theta = 0$ and an asymptotic expansion,
 $$K_{k, j}(z,z) = k^m \dashint_{\T} e^{- i k E (\theta + i \log (\| z\|^2/E) ) } e^{ k E e^{i \theta} } d\theta \simeq k^{m-1/2} \left( \frac{e\cdot \|z\|^2}{E} \right)^{kE} (2\pi E)^{-1/2}. $$
where we used the stationary phase formula for $d \theta$ integral. The result agrees with the exact one after applying Stirling formula.

The statement about the maximum of $B_{k,j}(z)$ can be obtained by solving
\[ \frac{d}{d|z|^2} \left(\frac{1}{k} \log B_{k,j}(z) \right) = -1 + E/\|z\|^2 = 0 \]
Indeed, the maximum of $B_{k,j}(z)$  occurs when $\|z\|^2=E$. 
 \end{proof}
 
Now we scale the equivariant Bargmann-Fock kernels around $H^{-1}(E)$
and prove Theorem \ref{EQUIVINT} in this case.
 Let $z_0 \in H^{-1}(E)$, i.e.
$\|z_0\|^2  = E$ and fix $u \in \R$. 
\begin{equation}\label{BFESCb} \Pi_{k, j} (z_0 (1+ \frac{u}{\sqrt{k}}), z_0 (1+ \frac{u}{\sqrt{k}})) = k^{m}
\dashint_{\T} e^{- i k E \theta} e^{ k \left( e^{i \theta } \|z_0\|^2  (1 + \frac{u}{\sqrt{k}})^2 -    \|z_0\|^2  (1 + \frac{u}{\sqrt{k}})^2 \right)} d\theta
\end{equation}
As  $k \to \infty$,
$$ e^{i \theta } \|z_0\|^2  (1 + \frac{u}{\sqrt{k}})^2 -    \|z_0\|^2  (1 + \frac{u}{\sqrt{k}})^2 = E (i \theta - \theta^2/2 + e_3(\theta))(1+ 2u/\sqrt{k} + u^2/k))$$ 
so the phase has the form $k E \Psi$ with
$$\begin{array}{l} \Psi =  i \theta   \left(  2u/\sqrt{k} + u^2/k \right)  -  \frac{\theta^2}{2} \left(1+2 u/\sqrt{k} + u^2/k \right) +  
 e_3(i \theta) \left(1+ 2u/\sqrt{k} + u^2/k \right), \end{array}$$ 
 where $e^x = 1 + x + x^2/2! + e_3(x)$.
We localize around $\theta = 0$ using a cutoff $\chi \in C_0^{\infty}(-1,1)$ and change variables $\theta \to k^{-1/2}  \theta$ to get
$$k^{-1/2} (2\pi)^{-1} \int_{\R} \chi(\theta/\sqrt{k})\; e^{i  \theta (2uE) - E \frac{\theta^2}{2}} A_k(\theta, u) d \theta,$$
where $A$ is a semi-classical symbol of order zero. Here, we absorbed the other terms, 
$$e^{ E( \frac{i \theta |u|^2}{\sqrt{k}} - \frac{\theta^2 u}{\sqrt{k}} - \frac{i \theta^3}{3! \sqrt{k}}) + O(1/k)}$$
into $A$. Since $A_0(\theta, u) = 1$
as $k \to \infty$ the integral tends to 
$$k^{-1/2} (2\pi)^{-1}\int_{\R} \; e^{i \theta (2 E u) - E \frac{\theta^2}{2}}  d \theta =(2 \pi k E)^{-1/2} e^{- 2 u^2 E  } (1 + O(k^{-1/2})), $$ 
Thus
\[ \Pi_{k, j} (z_0 (1+ \frac{u}{\sqrt{k}}), z_0 (1+ \frac{u}{\sqrt{k}})) = k^{m-1/2}\left( \frac{e^{-2 E u^2}}{\sqrt{2\pi E}} + O(k^{-1/2})\right)
\]
proving Theorem \ref{EQUIVINT}  in this case.

\section{Appendix  on off-diagonal decay estimates}

\begin{theo} (See Theorem 2 of \cite{Del} and Proposition 9 of \cite{L})] \label{AGMON1}  Let $M$ be a compact
\kahler manifold, and let $(L, h) \to M$ be a positive Hermitian line
bundle.  Then the exists a constant $\beta=\beta(M,L,h)>0$ such that
$$|\tilde\Pi_N(x, y)|_{\tilde h^N} \leq  CN^{m} e^{-\beta \sqrt{N} {d} (x, y)}.  $$
where ${d}(x,y)$ is the Riemannian distance with respect to
the \kahler metric $\tilde{\omega}$.\end{theo} 

The theorem is stated for strictly pseudo-convex domains in $\C^n$ but applies with no essential change to unit codisc bundles of positive Hermitian line bundles.

%
%
%
%
%
%
%
%
%

\end{document}